\documentclass[11pt, twoside, leqno]{article}

\usepackage{amssymb}
\usepackage{amsmath}
\usepackage{amsthm}
\usepackage{color}
\usepackage{mathrsfs}
\usepackage{txfonts}

\usepackage{indentfirst}

\allowdisplaybreaks

\newtheorem{theorem}{Theorem}[section]
\newtheorem{lemma}[theorem]{Lemma}

\theoremstyle{definition}
\newtheorem{remark}[theorem]{Remark}

\renewcommand{\appendix}{\par
   \setcounter{section}{0}%
   \setcounter{subsection}{0}%
   \setcounter{subsubsection}{0}%
   \gdef\thesection{\@Alph\c@section}%
   \gdef\thesubsection{\@Alph\c@section.\@arabic\c@subsection}%
   \gdef\theHsection{\@Alph\c@section.}%
   \gdef\theHsubsection{\@Alph\c@section.\@arabic\c@subsection}%
   \csname appendixmore\endcsname
 }

\numberwithin{equation}{section}

\usepackage{graphicx} 

\begin{document}

\arraycolsep=1pt

\title{\bf\Large Numerical approximations and error analysis of the Cahn-Hilliard equation with reaction rate dependent dynamic boundary conditions
\footnotetext{\hspace{-0.35cm} 2010 {\it
Mathematics Subject Classification}. 65M12; 65M06; 65N12; 65M22.
\endgraf {\it Key words and phrases.}
Cahn-Hilliard equation; Dynamic boundary conditions; Error estimates; Linear numerical
scheme; Energy stability.
}}
\author{Xuelian Bao\footnote{Corresponding author, School of Mathematical Sciences, Beijing Normal University, Beijing 100875, China (e-mail: xlbao@mail.bnu.edu.cn).}, Hui Zhang\footnote{Laboratory of Mathematics and Complex Systems, Ministry of Education and School of Mathematical Sciences, Beijing Normal University, Beijing 100875, China.}}
\date{}
\maketitle

\vspace{-0.8cm}

\begin{center}
\begin{minipage}{13cm}
{\small {\bf Abstract}\quad
We consider numerical approximations and error analysis for the Cahn-Hilliard equation with reaction rate dependent dynamic boundary conditions (P. Knopf et al., arXiv, 2020).
Based on the stabilized linearly implicit approach, a first-order in time, linear and energy stable scheme for solving this model is proposed.  The corresponding semi-discretized-in-time error estimates for the scheme are also derived.
Numerical experiments, including the simulations with different energy potentials, the comparison with the former work, the convergence results for the relaxation parameter $K\rightarrow0$ and $K\rightarrow\infty$ and the accuracy tests with respect to the time step size, are performed to validate the accuracy of
the proposed scheme and the error analysis.
}
\end{minipage}
\end{center}

\section{Introduction\label{s1}}
The Cahn-Hilliard equation, first introduced in \cite{CH1958}, was originally utilized to describe the phase separation and de-mixing processes of binary mixtures.
The standard Cahn-Hilliard equation can be written as follows:
\begin{equation}\label{CH}
\left\{
\begin{aligned}
&\phi_t=\Delta\mu, &\mbox{in}\  \Omega\times(0,T),\\
&\mu=-\varepsilon\Delta\phi+\frac{1}{\varepsilon}F'(\phi), &\mbox{in}\ \Omega\times(0,T),
\end{aligned}
\right.
\end{equation}
where the parameter $\varepsilon>0$, $\Omega\subseteq\mathbb{R}^{d}$ ($d=2,3$) denotes a bounded domain whose boundary $\Gamma=\partial\Omega$ with the unit outer vector field $\mathbf{n}$.
The function $\phi$
denotes the difference of two local relative concentrations, in order to describe the binary alloys. The regions with $\phi=\pm1$ in the domain $\Omega$ correspond to the pure phases of the materials, which are separated by a interfacial region whose thickness is proportional to $\varepsilon$.

In the Cahn-Hilliard equation, $\mu$ denotes the chemical potential in $\Omega$, which can be expressed as the Fr\'{e}chet derivative of the bulk free energy:
\begin{equation}\label{Ebulk}
E^{bulk}(\phi)=\int_{\Omega}\frac{\varepsilon}2|\nabla\phi|^2+\frac{1}{\varepsilon}F(\phi)\mbox{d}x,
\end{equation}
where $F$ denotes the potential in $\Omega$. The classical choice of $F$ is the smooth double-well potential
\begin{equation}\label{classicalF}
F(x)=\frac{1}4(x^2-1)^2, \qquad x\in \mathbb{R},
\end{equation}
which has a double-well structure with two minima at -1 and 1 and a local unstable maximum at 0.

Since the time-evolution of $\phi$ is confined in a bounded domain, suitable boundary conditions are needed.
The classical choice is the homogeneous Neumann conditions:
\begin{equation}\label{BCmu}
\partial_{\mathbf{n}}\mu=0, \quad \mbox{on}\  \Gamma\times(0,T),
\end{equation}
\begin{equation}\label{BCphi}
\partial_{\mathbf{n}}\phi=0, \quad \mbox{on}\  \Gamma\times(0,T),
\end{equation}
where $\partial_{\mathbf{n}}$ represents the outward normal derivative on $\Gamma$.
%
Obviously, the mass conservation law holds in the
bulk (i.e., in $\Omega$) with the no-flux boundary condition \eqref{BCmu}:
\begin{equation}
\int_{\Omega}\phi(t) \mbox{d}x=\int_{\Omega}\phi(0) \mbox{d}x, \quad t\in[0,T].
\end{equation}
In addition, the time evolution of the bulk
free energy $E^{bulk}$ (Eq. \eqref{Ebulk}) is decreasing with the boundary conditions \eqref{BCmu} and \eqref{BCphi}, namely,
\begin{equation}
\frac{d}{dt} E^{bulk}(\phi(t))+\int_{\Omega}|\nabla\mu|^2 \mbox{d}x=0, \quad t\in(0,T).
\end{equation}

When some particular applications (for instance, the hydrodynamic applications such as contact line problems) are taken into consideration, it's necessary to describe the short-range interactions between the mixture and the solid wall. However, the standard homogeneous Neumann conditions neglect the effects of the boundary to the bulk dynamics. Thus, several dynamic boundary conditions have been proposed and analysed in recent years, see for instance, (\cite{Racke2003}, \cite{wu2004}, \cite{Gal2006}, \cite{GMS2011}, \cite{colli2015}, \cite{colli2017}, \cite{Mininni2017}, \cite{knopf2019}, \cite{liuwu2019}, \cite{knopf2020}). These dynamic boundary conditions are based on the system with added surface free energy (\cite{fischer1997}, \cite{fischer1998}, \cite{Kenzler2001}).
The total free energy can be written as
\begin{equation}\label{Etotal}
E^{total}(\phi)=E^{bulk}(\phi)+E^{surf}(\phi),
\end{equation}
\begin{equation}\label{Esurf}
E^{surf}(\phi)=\int_{\Gamma}\frac{\delta\kappa}2|\nabla_{\Gamma}\phi|^2+\frac{1}{\delta}G(\phi)\mbox{d}S,
\end{equation}
where $\nabla_{\Gamma}$ represents the tangential or surface gradient operator on $\Gamma$, $G$ is the surface potential, $\delta$ denotes the thickness of the interfacial region on $\Gamma$ and the parameter $\kappa$ is related to the surface diffusion. When $\kappa=0$, it is related to the moving contact line problem \cite{Thompson1989}.

In the present work, we summarize three Cahn-Hilliard models with dynamic boundary conditions in detail. All the dynamic boundary conditions of the three models have a Cahn-Hilliard type structure. And they can be interpreted as an $H^{-1}$-gradient flow of the total free energy.

The first Cahn-Hilliard model with dynamic boundary conditions was proposed by G.R. Goldstein, A. Miranville, and G. Schimperna \cite{GMS2011}:
\begin{equation}\label{GMS}
\left\{
\begin{aligned}
&\phi_t=\Delta\mu, &\mbox{in}\  \Omega\times(0,T),\\
&\mu=-\varepsilon\Delta\phi+\frac{1}{\varepsilon}F'(\phi), &\mbox{in}\ \Omega\times(0,T),\\
&\phi|_{\Gamma}=\psi, &\mbox{on}\  \Gamma\times(0,T),\\
&\psi_t=\Delta_{\Gamma}\mu-\partial_{\mathbf{n}}\mu, &\mbox{on}\  \Gamma\times(0,T),\\
&\mu=-\delta\kappa\Delta_{\Gamma}\psi+\frac{1}{\delta}G'(\psi)+\varepsilon\partial_{\mathbf{n}}\phi
&\mbox{on}\  \Gamma\times(0,T).
\end{aligned}
\right.
\end{equation}
In the present work, we denote the model as the GMS model for convenience. Here, $\Delta_{\Gamma}$ denotes the Laplace-Beltrami operator on $\Gamma$. Note that
the chemical potentials in the bulk and on the boundary are the same. Moreover, the dynamic boundary conditions ensure the conservation of the total mass (namely, the sum of the bulk and boundary mass):
\begin{equation}\label{GMSmassconservation}
\int_{\Omega}\phi(t) \mbox{d}x+\int_{\Gamma}\psi(t) \mbox{d}S=\int_{\Omega}\phi(0) \mbox{d}x +\int_{\Gamma}\psi(0) \mbox{d}S, \quad \mbox{for all} \  t\in[0,T],
\end{equation}
and the energy dissipation law:
\begin{equation}\label{GMSenergylaw}
\frac{\mbox{d}}{\mbox{d}t} E^{total}(\phi,\psi)=-\|\nabla\mu\|_{\Omega}^2
-\|\nabla_{\Gamma}\mu\|_{\Gamma}^2\leq 0.
\end{equation}

The second Cahn-Hilliard model with dynamic boundary conditions was proposed by C. Liu and H. Wu \cite{liuwu2019}:
\begin{equation}\label{CHLW}
\left\{
\begin{aligned}
&\phi_t=\Delta\mu, &\mbox{in}\  \Omega\times(0,T),\\
&\mu=-\varepsilon\Delta\phi+\frac{1}{\varepsilon}F'(\phi), &\mbox{in}\ \Omega\times(0,T),\\
&\partial_{\mathbf{n}}\mu=0, &\mbox{on}\  \Gamma\times(0,T),\\
&\phi|_{\Gamma}=\psi, &\mbox{on}\  \Gamma\times(0,T),\\
&\psi_t=\Delta_{\Gamma}\mu_{\Gamma}, &\mbox{on}\  \Gamma\times(0,T),\\
&\mu_{\Gamma}=-\delta\kappa\Delta_{\Gamma}\psi+\frac{1}{\delta}G'(\psi)+\varepsilon\partial_{\mathbf{n}}\phi
&\mbox{on}\  \Gamma\times(0,T).
\end{aligned}
\right.
\end{equation}
We denote it as the Liu-Wu model for short. Here, $\mu_{\Gamma}$ denotes the chemical potential on the boundary.
The model assumes that there is no mass exchange between the bulk and the boundary, namely, $\partial_{\mathbf{n}}\mu=0$.
Different from the GMS model ($\mu=\mu_{\Gamma}$), the chemical potential $\mu$ and $\mu_{\Gamma}$ are not directly coupled. Similarly, we can obtain the following mass conservation law:
\begin{equation}\label{LWmassconservation}
\int_{\Omega}\phi(t) \mbox{d}x=\int_{\Omega}\phi(0) \mbox{d}x \quad \mbox{and} \quad \int_{\Gamma}\psi(t) \mbox{d}S=\int_{\Gamma}\psi(0) \mbox{d}S, \quad t\in[0,T],
\end{equation}
indicating that the Liu-Wu model satisfies the mass conservation law in the bulk and on the boundary respectively. Moreover, the energy dissipation law \eqref{GMSenergylaw} also holds for the Liu-Wu model. The readers can find the well-posedness results for the Liu-Wu model and the GMS model in \cite{liuwu2019} and \cite{GMS2011} respectively.

Recently, Knopf et al. \cite{knopf2020} proposed a new model, which can be interpreted as an  interpolation between the Liu-Wu model and the GMS model. It reads as follows,
\begin{equation}\label{CHK}
\left\{
\begin{aligned}
&\phi_t=\Delta\mu, \qquad &\mbox{in}\  \Omega\times(0,T),\\
&\mu=-\varepsilon\Delta\phi+\frac{1}{\varepsilon}F'(\phi), \qquad &\mbox{in}\ \Omega\times(0,T),\\
&K\partial_{\mathbf{n}}\mu=\mu_{\Gamma}-\mu, \qquad &\mbox{on}\  \Gamma\times(0,T),\\
&\phi|_{\Gamma}=\psi, \qquad &\mbox{on}\  \Gamma\times(0,T),\\
&\psi_t=\Delta_{\Gamma}\mu_{\Gamma}-\partial_{\mathbf{n}}\mu, \qquad &\mbox{on}\  \Gamma\times(0,T),\\
&\mu_{\Gamma}=-\delta\kappa\Delta_{\Gamma}\psi+\frac{1}{\delta}G'(\psi)+\varepsilon\partial_{\mathbf{n}}\phi,
\qquad &\mbox{on}\  \Gamma\times(0,T).
\end{aligned}
\right.
\end{equation}
In the present work, we use the authors' initials and refer it to be the KLLM model for convenience.
Here, in order to describe the binary alloys, $\phi$ and $\psi$ represent the phase-field order parameter or the concentration of one material component in the bulk and on the boundary, respectively.
$\mu$ and $\mu_{\Gamma}$ represent the chemical potentials in $\Omega$ and on $\Gamma$, respectively.
Notice that $\mu$ and $\mu_{\Gamma}$
are coupled by the Robin type boundary condition $K\partial_{\mathbf{n}}\mu=\mu_{\Gamma}-\mu$, where the positive parameter $K$ is the
relaxation parameter.
The equation on the boundary ($\eqref{CHK}_4$) can be viewed as a chemical reaction in a general case since it describes that one species ($\phi$) changes into another species ($\psi$) on the boundary. And $\eqref{CHK}_3$ means that there exists mass transfer between the bulk ($\phi$) and the boundary ($\psi$). Thus, the constant $1/K$ can be interpreted as the reaction rate.
%
%
%
The well-posedness of the system \eqref{CHK} and convergence to the Liu-Wu model (as $K\rightarrow\infty$) and the GMS model (as $K\rightarrow0$) in both the weak and the strong sense have been investigated by Knopf et al. \cite{knopf2020}.

The numerical approximations of the Cahn-Hilliard equation and its variants have already been well investigated. There exists extensive efficient techniques for the time discretization, such as the stabilized linearly implicit approach \cite{he2007}, the convex splitting approach  (\cite{shen2012}, \cite{grun2013}), the invariant energy quadratization (IEQ) method (\cite{ieq1}, \cite{ieq2}, \cite{zhaojia2018}) and the scalar auxiliary variable (SAV) method \cite{sav}. For the higher order scheme and more general case of the phase-field models, we refer the readers to the recent work of Gong et al. \cite{gong2020}. Moreover, X. Yang et al. have proposed efficient numerical schemes on the phase-field models with more complicated potentials (the logarithmic Flory-Huggins potential \cite{flory} and the nonlocal potential \cite{yangnonlocal}). Recently, there have been numerical approximations for the Cahn-Hilliard equation with dynamic boundary conditions ( see for instance, \cite{bao2020}, \cite{Cherfils2010}, \cite{Cherfils2014}, \cite{Israel2015}, \cite{Fukao2017} and \cite{bachelor}).
Specifically, for the Liu-Wu model,
the finite element scheme has been proposed in \cite{bachelor} and \cite{Garcke2020},
where the straightforward discretization based on piecewise linear finite element functions was utilized to simulate the model, and the corresponding nonlinear system was solved by Newton's method. A recent contribution on the numerical analysis can be found in \cite{Metzger2019}. For the KLLM model, we refer the readers to \cite{knopf2020} for the finite element numerical approximations and numerical analysis.
However, the backward implicit Euler method was used for time discretization in the finite element schemes mentioned above, where one needs to solve nonlinear systems at each time step. Recently, based on the stabilized linearly implicit approach, a linear and energy stable numerical scheme has been proposed for the Liu-Wu model \cite{bao2020} and the corresponding semi-discrete-in-time error estimates are carried out.

Inspired by the numerical scheme in \cite{bao2020}, a first-order in time, linear and energy stable scheme for solving the KLLM model is proposed in the present work.
%
Note that the scheme is highly efficient since one only needs to solve a linear equation at each time step. Numerical simulations are performed in the two-dimensional space to validate the accuracy and stability of the scheme.
We also investigate the error estimates in semi-discrete-in-time for the scheme.
To the best of the authors' knowledge, the proposed scheme is the first linear numerical scheme to solve the KLLM model and it is the first work to give the corresponding semi-discrete-in-time error estimates.

The rest of the paper is organized as follows.
We first present some notions and notations appearing in this article in Section 2. In Section 3, the stabilized scheme for the KLLM model and the energy stability are derived. The error estimates are constructed in Section 4. In Section 5, we present the numerical examples and illustrate the convergence results for $K\rightarrow0$ and $K\rightarrow\infty$. The accuracy tests are also displayed in this section. Finally, the conclusion is presented in Section 6.

\section{Preliminaries \label{s2}}

Before giving the stabilized scheme and the corresponding error analysis, we make some definitions in this section.

We consider a finite time interval $[0,T]$ and a domain $\Omega\subset\mathbb{R}^d$ ($d=2, 3$), which is a bounded domain with sufficient smooth boundary $\Gamma=\partial\Omega$ and $\mathbf{n}=\mathbf{n}(x)$ is the unit outer normal vector on $\Gamma$. In this article, we need the boundary $\Gamma$ to be of class $C^{k,1}$ with $k\geq 3$. This regularity is needed for the error estimates in Section 4.

The norm and inner product of $L^2(\Omega)$ and $L^2(\Gamma)$ are denoted by $\|\cdot\|_{\Omega}$, $(\cdot,\cdot)_{\Omega}$ and $\|\cdot\|_{\Gamma}$, $(\cdot,\cdot)_{\Gamma}$ respectively. The usual norm in $H^k(\Omega)$ and $H^k(\Gamma)$ are denoted by $\|\cdot\|_{H^k(\Omega)}$ and $\|\cdot\|_{H^k(\Gamma)}$ respectively.

Let $\tau$ be the time step size. For a sequence of functions $f^0, f^1, \ldots, f^N$ in some Hilbert space $E$, we denote the sequence by $\{f_{\tau}\}$ and define the following discrete norm for $\{f_{\tau}\}$:
\begin{equation}\label{disnorm}
\|f_{\tau}\|_{l^{\infty}(E)}=\max_{0\leq n\leq N}\bigg{(}\|f^n\|_E\bigg{)}.
\end{equation}
We denote by $C$ a generic constant that is independent of $\tau$ but possibly depends on the parameters and solutions, and use $f\lesssim g$ to say that there is a generic constant $C$ such that $f\leqslant C g$.

\section{ The Cahn-Hilliard equation with reaction rate dependent dynamic boundary conditions and its numerical scheme}\label{s3}

In this section, we first summarize the mass conservation and the energy dissipation law of the KLLM model. Then we propose the stabilized linear numerical scheme and prove the discrete energy dissipation law.

Since $\phi$ is the phase-field order parameter in the bulk, denote its trace $\phi|_{\Gamma} \triangleq\psi$ as the order parameter on the boundary.
In the bulk $\Omega$, assume that $\phi$ is a locally conserved quantity that satisfies the continuity equation
\begin{equation}\label{phice}
\phi_t+\nabla\cdot(\phi\mathbf{u})=0, \qquad (x,t)\in \Omega\times(0,T),
\end{equation}
where $\mathbf{u}$ is the microscopic effect velocity.

We assume that there exists mass exchange between the bulk $\Omega$ and the boundary $\Gamma$, which is denoted by the flux $J=\phi\mathbf{u}$.
Assume that the mass flux is directly driven by differences between the chemical potentials in the sense that
\begin{equation}\label{uBC}
K(J\cdot\mathbf{n})=K(\phi\mathbf{u}\cdot\mathbf{n})=\mu-\mu_\Gamma, \qquad (x,t)\in\Gamma\times(0,T),
\end{equation}
where $K$ is a positive parameter describing the extent of mass exchange.
Eq. \eqref{uBC} is the boundary condition of $\mathbf{u}$.

Assume that the boundary dynamics is characterized by a local mass conservation law analogous to \eqref{phice}, such that
\begin{equation}\label{psice}
\psi_t+\nabla_{\Gamma}\cdot(\psi\mathbf{v})-J\cdot\mathbf{n}=0, \qquad (x,t)\in \Gamma\times(0,T),
\end{equation}
where $\mathbf{v}$ denotes the microscopic effective tangential velocity field on the boundary $\Gamma$. Assume that $\Gamma$ is a closed manifold, thus, there is no need to impose any boundary condition on $\mathbf{v}$.

The mass is conserved in the sense that
\begin{equation}\label{massconservation}
\int_{\Omega}\phi(t)\mbox{d}x+\int_{\Gamma}\psi(t)\mbox{d}S=
\int_{\Omega}\phi(0)\mbox{d}x+\int_{\Gamma}\psi(0)\mbox{d}S, \quad \forall t\in [0,T].
\end{equation}
To this end, integrating \eqref{phice} over $\Omega$, we have
\begin{equation}\label{mass1}
\frac{d}{dt}\int_{\Omega}\phi(\cdot,t)\mbox{d}x+\int_{\Gamma} \phi
\mathbf{u}\cdot\mathbf{n}\mbox{d}S=0, \forall t\in(0,T),
\end{equation}
and integrating \eqref{psice} over $\Gamma$, we have
\begin{equation}\label{mass2}
\frac{d}{dt}\int_{\Gamma}\psi(\cdot,t)\mbox{d}S-\int_{\Gamma} J \cdot\mathbf{n}\mbox{d}S=0, \forall t\in(0,T).
\end{equation}
Combining \eqref{mass1} with \eqref{mass2} and the flux $J=\phi
\mathbf{u}$, we obtain the total mass conservation law, see \eqref{massconservation}.

%
%

Then we show the energy law of the KLLM model, where the total free energy (sum of the bulk and surface free energies) is decreasing in time. Precisely, multiplying the first equation of \eqref{CHK} by $\mu$ and integrating over $\Omega$, we get
$$
(\phi_t, \mu)_{\Omega}=(\Delta \mu, \mu)_{\Omega}=(\partial_{\mathbf{n}}\mu, \mu)_{\Gamma}-\|\nabla\mu\|_{L^2(\Omega)}^2.
$$
Since
$$(\phi_t, \mu)_{\Omega}=(\phi_t,-\varepsilon\Delta\phi+\frac{1}{\varepsilon}F'(\phi))_{\Omega},$$
$$
(\phi_t,-\varepsilon\Delta\phi)_{\Omega}=-(\varepsilon\partial_{\mathbf{n}}\phi, \phi_t)_{\Gamma}+\frac{\varepsilon}{2}\frac{\mbox{d}}{\mbox{d}t}(\int_{\Omega} |\nabla\phi|^2\mbox{d}x),
$$
$$(\phi_t, \frac{1}{\varepsilon}F'(\phi))_{\Omega}=\frac{\mbox{d}}{\mbox{d}t}(\int_{\Omega} \frac{1}{\varepsilon} F(\phi)\mbox{d}x),$$
we arrive that
\begin{equation}\label{energy1}
\frac{\mbox{d}}{\mbox{d}t}(\int_{\Omega}\frac{1}{\varepsilon} F(\phi)\mbox{d}x+\frac{\varepsilon}{2}\int_{\Omega} |\nabla\phi|^2\mbox{d}x)-(\varepsilon\partial_{\mathbf{n}}\phi, \phi_t)_{\Gamma}=(\partial_{\mathbf{n}}\mu, \mu)_{\Gamma}-\|\nabla\mu\|_{L^2(\Omega)}^2.
\end{equation}
Multiplying the boundary equation in \eqref{CHK} by $\mu_{\Gamma}$ and integrating over $\Gamma$, we get
$$
(\psi_t, \mu_{\Gamma})_{\Gamma}=(\Delta_{\Gamma} \mu_{\Gamma}, \mu_{\Gamma})_{\Gamma}-(\partial_{\mathbf{n}}\mu, \mu_{\Gamma})_{\Gamma}=-\|\nabla_{\Gamma}\mu_{\Gamma}\|_{L^2(\Gamma)}^2-(\partial_{\mathbf{n}}\mu, \mu_{\Gamma})_{\Gamma}.
$$
Since
$$(\psi_t, \mu_{\Gamma})_{\Gamma}=(\psi_t,-\delta\kappa\Delta_{\Gamma}\psi+\frac{1}{\delta}G'(\psi)+\varepsilon\partial_{\mathbf{n}}\phi)_{\Gamma},$$
$$
(\psi_t,-\delta\kappa\Delta_{\Gamma}\psi)_{\Gamma}=\frac{\delta\kappa}{2}\frac{\mbox{d}}{\mbox{d}t}(\int_{\Gamma} |\nabla_{\Gamma}\psi|^2\mbox{d}S),
$$
$$(\psi_t, \frac{1}{\delta}G'(\psi))_{\Gamma}=\frac{\mbox{d}}{\mbox{d}t}(\int_{\Gamma} \frac{1}{\delta} G(\psi)\mbox{d}S),$$
we arrive that
\begin{equation}\label{energy2}
\frac{\mbox{d}}{\mbox{d}t}(\int_{\Gamma} \frac{1}{\delta} G(\psi)\mbox{d}S +\frac{\delta\kappa}{2}\int_{\Gamma} |\nabla_{\Gamma}\psi|^2\mbox{d}S)+(\varepsilon\partial_{\mathbf{n}}\phi, \psi_t)_{\Gamma}=-(\partial_{\mathbf{n}}\mu, \mu_{\Gamma})_{\Gamma}-\|\nabla_{\Gamma}\mu_{\Gamma}\|_{L^2(\Gamma)}^2.
\end{equation}
Adding \eqref{energy1} and \eqref{energy2} together, we get
\begin{equation}\label{energy3}
\begin{aligned}
&\frac{\mbox{d}}{\mbox{d}t} (\int_{\Omega}\frac{1}{\varepsilon} F(\phi)+\frac{\varepsilon}{2} |\nabla\phi|^2\mbox{d}x+\int_{\Gamma} \frac{1}{\delta} G(\psi)+\frac{\delta\kappa}{2}|\nabla_{\Gamma}\psi|^2\mbox{d}S)\\
&=-\|\nabla\mu\|_{L^2(\Omega)}^2
-\|\nabla_{\Gamma}\mu_{\Gamma}\|_{L^2(\Gamma)}^2+(\partial_{\mathbf{n}}\mu, \mu-\mu_{\Gamma})_{\Gamma}\\
&=-\|\nabla\mu\|_{L^2(\Omega)}^2
-\|\nabla_{\Gamma}\mu_{\Gamma}\|_{L^2(\Gamma)}^2 -K\|\partial_{\mathbf{n}}\mu\|_{L^2(\Gamma)}^2.
\end{aligned}
\end{equation}
Since $K>0$, we arrive at
$$\frac{\mbox{d}}{\mbox{d}t} (\int_{\Omega}\frac{1}{\varepsilon} F(\phi)+\frac{\varepsilon}{2} |\nabla\phi|^2\mbox{d}x+\int_{\Gamma} \frac{1}{\delta} G(\psi)+\frac{\delta\kappa}{2}|\nabla_{\Gamma}\psi|^2\mbox{d}S)\leq 0,$$
namely,
$$\frac{\mbox{d}}{\mbox{d}t} [E^{bulk}(\phi)+E^{surf}(\psi)]\leq 0.$$

Now we present the numerical scheme for the KLLM model (namely, Eq. \eqref{CHK}).
The scheme can be written as follows,
\begin{eqnarray}\label{SIscheme1}
&&\frac{\phi^{n+1}-\phi^{n}}{\tau}=\Delta\mu^{n+1}, \quad \mbox{in}\  \Omega,\\
\label{SIscheme2}
&&\mu^{n+1}=-\varepsilon\Delta\phi^{n+1}+\frac{1}{\varepsilon}F'(\phi^{n})+s_1(\phi^{n+1}-\phi^{n}), \quad \mbox{in}\ \Omega,\\
\label{SIscheme3}
&&K\partial_{\mathbf{n}}\mu^{n+1}=\mu_{\Gamma}^{n+1}-\mu^{n+1}, \quad \mbox{on}\  \Gamma,\\
\label{SIscheme4}
&&\phi^{n+1}|_{\Gamma}=\psi^{n+1}, \quad \mbox{on}\  \Gamma,\\
\label{SIscheme5}
&&\frac{\psi^{n+1}-\psi^{n}}{\tau}=\Delta_{\Gamma}\mu_{\Gamma}^{n+1}-\partial_{\mathbf{n}}\mu^{n+1}, \quad \mbox{on}\  \Gamma,\\
\label{SIscheme6}
&&\mu_{\Gamma}^{n+1}=-\delta\kappa\Delta_{\Gamma}\psi^{n+1}+\frac{1}{\delta}G'(\psi^{n})
+\varepsilon\partial_{\mathbf{n}}\phi^{n+1}+s_2(\psi^{n+1}-\psi^{n}), \quad \mbox{on}\  \Gamma.
\end{eqnarray}
Here, $T$ is an arbitrary and fixed time, $N$ is the number of time steps and $\tau=T/N$ is the step size.

\begin{remark}
The parameters $s_1, s_2>0$. And the stabilization terms $s_1(\phi^{n+1}-\phi^{n})$ and $s_2(\psi^{n+1}-\psi^{n})$ are added in the bulk and on the boundary to enhance the stability, respectively.
\end{remark}

\begin{remark}
For the Liu-Wu model, we need to modify Eq. \eqref{SIscheme3} to be
$$
\partial_{\mathbf{n}}\mu^{n+1}=0, \quad \mbox{on}\  \Gamma,
$$
and the last term $\partial_{\mathbf{n}}\mu^{n+1}$ in \eqref{SIscheme5} vanishes.
In this article, the scheme for the Liu-Wu model reads as follows, which is the same as that in \cite{bao2020}:
\begin{eqnarray}\label{liuwuscheme1}
&&\frac{\phi^{n+1}-\phi^{n}}{\tau}=\Delta\mu^{n+1}, \quad \mbox{in}\  \Omega,\\
&&\mu^{n+1}=-\varepsilon\Delta\phi^{n+1}+\frac{1}{\varepsilon}F'(\phi^{n})+s_1(\phi^{n+1}-\phi^{n}), \quad \mbox{in}\ \Omega,\\
\label{liuwuscheme2}
&&\partial_{\mathbf{n}}\mu^{n+1}=0, \quad \mbox{on}\  \Gamma,\\
\label{liuwuscheme3}
&&\phi^{n+1}|_{\Gamma}=\psi^{n+1}, \quad \mbox{on}\  \Gamma,\\
\label{liuwuscheme4}
&&\frac{\psi^{n+1}-\psi^{n}}{\tau}=\Delta_{\Gamma}\mu_{\Gamma}^{n+1}, \quad \mbox{on}\  \Gamma,\\
\label{liuwuscheme5}
&&\mu_{\Gamma}^{n+1}=-\delta\kappa\Delta_{\Gamma}\psi^{n+1}+\frac{1}{\delta}G'(\psi^{n})
+\varepsilon\partial_{\mathbf{n}}\phi^{n+1}+s_2(\psi^{n+1}-\psi^{n}), \quad \mbox{on}\  \Gamma.
\label{liuwuscheme6}
\end{eqnarray}

For the GMS model, Eq. \eqref{SIscheme3} is modified to be
$$
\mu^{n+1}|_{\Gamma}=\mu_{\Gamma}^{n+1}, \quad \mbox{on}\  \Gamma.
$$
In this article, the scheme for the GMS model reads as follows,
\begin{eqnarray}\label{GMSscheme1}
&&\frac{\phi^{n+1}-\phi^{n}}{\tau}=\Delta\mu^{n+1}, \quad \mbox{in}\  \Omega,\\
\label{GMSscheme2}
&&\mu^{n+1}=-\varepsilon\Delta\phi^{n+1}+\frac{1}{\varepsilon}F'(\phi^{n})+s_1(\phi^{n+1}-\phi^{n}), \quad \mbox{in}\ \Omega,\\
\label{GMSscheme3}
&&\phi^{n+1}|_{\Gamma}=\psi^{n+1}, \quad \mbox{on}\  \Gamma,\\
\label{GMSscheme4}
&&\frac{\psi^{n+1}-\psi^{n}}{\tau}=\Delta_{\Gamma}\mu^{n+1}-\partial_{\mathbf{n}}\mu^{n+1}, \quad \mbox{on}\  \Gamma,\\
\label{GMSscheme5}
&&\mu^{n+1}=-\delta\kappa\Delta_{\Gamma}\psi^{n+1}+\frac{1}{\delta}G'(\psi^{n})
+\varepsilon\partial_{\mathbf{n}}\phi^{n+1}+s_2(\psi^{n+1}-\psi^{n}), \quad \mbox{on}\  \Gamma.
\end{eqnarray}

From the above schemes we can conclude that the limiting cases are included in the proposed scheme in a general sense. The proposed scheme is based on the stabilized linearly implicit approach and we use the same strategy to deal with the limit cases. Precisely, in the scheme for the Liu-Wu model and the GMS model, we deal with the linear terms implicitly and the nonlinear terms explicitly and the stabilization terms are used.

\end{remark}

We have the energy stability as follows.
\begin{theorem}\label{ener_stable}
If the parameters $s_1$ and $s_2$ satisfy
\begin{equation}\label{s1s2}
s_1\geq\frac{1}{2\varepsilon} \max_{\xi\in\mathbb{R}} F''(\xi), \  s_2\geq\frac{1}{2\delta} \max_{\eta\in\mathbb{R}} G''(\eta),
\end{equation}
the scheme (\ref{SIscheme1})-(\ref{SIscheme6}) is energy stable in the sense that
\begin{equation}
\frac{E(\phi^{n+1},\psi^{n+1})-E(\phi^{n},\psi^{n})}{\tau}\leq-\|\nabla\mu^{n+1}\|^2_{L^2(\Omega)}
-\|\nabla_{\Gamma}\mu_{\Gamma}^{n+1}\|^2_{L^2(\Gamma)}-\frac{1}{K}\|\mu^{n+1}-\mu_{\Gamma}^{n+1}\|^2_{L^2(\Gamma)},
\end{equation}
where
\begin{equation}
E(\phi^{n},\psi^{n})=\int_{\Omega}\frac{1}{\varepsilon} F(\phi^{n})+\frac{\varepsilon}{2} |\nabla\phi^{n}|^2\mbox{d}x+\int_{\Gamma} \frac{1}{\delta} G(\psi^{n})+\frac{\delta\kappa}{2}|\nabla_{\Gamma}\psi^{n}|^2\mbox{d}S
\end{equation}
\end{theorem}

\begin{proof}
By taking inner product of \eqref{SIscheme1} with $\mu^{n+1}$ in $\Omega$, we have
\begin{equation}\label{proof1}
(\frac{\phi^{n+1}-\phi^{n}}{\tau}, \mu^{n+1})_{\Omega}=(\Delta\mu^{n+1}, \mu^{n+1})_{\Omega}=(\partial_{\mathbf{n}}\mu^{n+1},\mu^{n+1})_{\Gamma}-\|\nabla\mu^{n+1}\|^2_{L^2(\Omega)}.
\end{equation}
For the boundary integral term, by using \eqref{SIscheme3}, we have
$$
(\partial_{\mathbf{n}}\mu^{n+1},\mu^{n+1})_{\Gamma}=
\frac{1}K(\mu^{n+1}_{\Gamma}-\mu^{n+1},\mu^{n+1})_{\Gamma}.
$$

By using \eqref{SIscheme2}, we have
\begin{equation}\label{proof2}
(\frac{\phi^{n+1}-\phi^{n}}{\tau}, \mu^{n+1})_{\Omega}=(\frac{\phi^{n+1}-\phi^{n}}{\tau}, -\varepsilon\Delta\phi^{n+1}
+\frac{1}{\varepsilon}F'(\phi^{n})+s_1(\phi^{n+1}-\phi^{n}))_{\Omega},
\end{equation}
and
\begin{equation}\label{proof3}
(\frac{\phi^{n+1}-\phi^{n}}{\tau}, -\varepsilon\Delta\phi^{n+1})_{\Omega}=
-\varepsilon(\partial_{\mathbf{n}}\phi^{n+1}, \frac{\phi^{n+1}-\phi^{n}}{\tau})_{\Gamma}+\varepsilon(\nabla\phi^{n+1},
\frac{\nabla\phi^{n+1}-\nabla\phi^{n}}{\tau})_{\Omega}.
\end{equation}

For the boundary integral term in \eqref{proof3}, by taking the inner product of \eqref{SIscheme5} with $\mu^{n+1}_{\Gamma}$ on $\Gamma$, we obtain
\begin{equation}\label{proof4}
\begin{aligned}
(\frac{\psi^{n+1}-\psi^{n}}{\tau}, \mu^{n+1}_{\Gamma})_{\Gamma}&=
(\Delta_{\Gamma} \mu^{n+1}_{\Gamma}, \mu^{n+1}_{\Gamma})_{\Gamma}- (\partial_{\mathbf{n}}\mu^{n+1},\mu^{n+1}_{\Gamma})_{\Gamma}\\
&=-\|\nabla_{\Gamma} \mu^{n+1}_{\Gamma}\|^2_{L^2(\Gamma)}
-(\partial_{\mathbf{n}}\mu^{n+1},\mu^{n+1}_{\Gamma})_{\Gamma}.
\end{aligned}
\end{equation}

By using \eqref{SIscheme6}, we have
\begin{equation}\label{proof5}
(\frac{\psi^{n+1}-\psi^{n}}{\tau}, \mu^{n+1}_{\Gamma})_{\Gamma}=(\frac{\psi^{n+1}-\psi^{n}}{\tau},-\delta\kappa\Delta_{\Gamma}
\psi^{n+1}+\frac{1}{\delta}G'(\psi^{n})+\varepsilon\partial_{\mathbf{n}}\phi^{n+1}+s_2(\psi^{n+1}-\psi^{n}))_{\Gamma},
\end{equation}
and
\begin{equation}\label{proof6}
(\frac{\psi^{n+1}-\psi^{n}}{\tau}, -\delta\kappa\Delta_{\Gamma}\psi^{n+1})_{\Gamma}=
(\frac{\nabla_{\Gamma}\psi^{n+1}-\nabla_{\Gamma}\psi^{n}}{\tau}, \delta\kappa\nabla_{\Gamma}\psi^{n+1})_{\Gamma}.
\end{equation}

To handle the nonlinear term associated with $F'$ and $G'$ in \eqref{proof2} and \eqref{proof5}, we need the following identities
\begin{equation}\label{proof7}
\begin{aligned}
F'(\phi^{n})(\phi^{n+1}-\phi^{n})&=F(\phi^{n+1})-F(\phi^{n})-\frac{F''(\eta)}2(\phi^{n+1}-\phi^{n})^2,\\
G'(\psi^{n})(\psi^{n+1}-\psi^{n})&=G(\psi^{n+1})-G(\psi^{n})-\frac{G''(\zeta)}2(\psi^{n+1}-\psi^{n})^2,
\end{aligned}
\end{equation}
with some $\eta\in(\phi^n, \phi^{n+1})$ and $\zeta\in(\psi^n, \psi^{n+1})$.

Combining the equations mentioned above, we get
$$
\begin{aligned}
&(\frac{\phi^{n+1}-\phi^{n}}{\tau}, \mu^{n+1})_{\Omega}+(\frac{\psi^{n+1}-\psi^{n}}{\tau}, \mu^{n+1}_{\Gamma})_{\Gamma}\\
&=(\partial_{\mathbf{n}}\mu^{n+1},\mu^{n+1})_{\Gamma}
-\|\nabla\mu^{n+1}\|^2_{L^2(\Omega)}-\|\nabla_{\Gamma}\mu^{n+1}_{\Gamma}\|^2_{L^2(\Gamma)}-(\partial_{\mathbf{n}}\mu^{n+1},\mu^{n+1}_{\Gamma})_{\Gamma}\\
&=-\|\nabla\mu^{n+1}\|^2_{L^2(\Omega)}-\|\nabla_{\Gamma}\mu^{n+1}_{\Gamma}\|^2_{L^2(\Gamma)}-\frac{1}K\|\mu^{n+1}-\mu_{\Gamma}^{n+1}\|^2_{L^2(\Gamma)},
\end{aligned}
$$
and
$$
\begin{aligned}
&(\frac{\phi^{n+1}-\phi^{n}}{\tau}, \mu^{n+1})_{\Omega}+(\frac{\psi^{n+1}-\psi^{n}}{\tau}, \mu^{n+1}_{\Gamma})_{\Gamma}\\
&=\varepsilon(\nabla\phi^{n+1},
\frac{\nabla\phi^{n+1}-\nabla\phi^{n}}{\tau})_{\Omega}
+\frac{1}{\varepsilon}(F'(\phi^{n}),\frac{\phi^{n+1}-\phi^{n}}{\tau})_{\Omega}+\frac{s_1}{\tau}\|\phi^{n+1}-\phi^{n}\|^2_{L^2(\Omega)}\\
&+(\delta\kappa\nabla_{\Gamma}\psi^{n+1},\frac{\nabla_{\Gamma}\psi^{n+1}-\nabla_{\Gamma}\psi^{n}}{\tau})_{\Gamma}+\frac{1}{\delta}(G'(\psi^{n}),\frac{\psi^{n+1}-\psi^{n}}{\tau})_{\Gamma}
+\frac{s_2}{\tau}\|\psi^{n+1}-\psi^{n}\|^2_{L^2(\Gamma)}\\
&=\varepsilon(\nabla\phi^{n+1},
\frac{\nabla\phi^{n+1}-\nabla\phi^{n}}{\tau})_{\Omega}+\frac{1}{\varepsilon}
(\frac{F(\phi^{n+1})-F(\phi^{n})}{\tau},1)_{\Omega}
-\frac{1}{2\varepsilon}(F''(\eta),\frac{(\phi^{n+1}-\phi^{n})^2}{\tau})_{\Omega}\\
&+\frac{s_1}{\tau}\|\phi^{n+1}-\phi^{n}\|^2_{L^2(\Omega)}
+\delta\kappa(\nabla_{\Gamma}\psi^{n+1},\frac{\nabla_{\Gamma}\psi^{n+1}-\nabla_{\Gamma}\psi^{n}}{\tau})_{\Gamma}
+\frac{1}{\delta}(\frac{G(\psi^{n+1})-G(\psi^{n})}{\tau},1)_{\Gamma}\\
&-\frac{1}{2\delta}(G''(\zeta),\frac{(\psi^{n+1}-\psi^{n})^2}{\tau})_{\Gamma}
+\frac{s_2}{\tau}\|\psi^{n+1}-\psi^{n}\|^2_{L^2(\Gamma)}\\
&=\frac{\varepsilon}{2\tau}(\|\nabla\phi^{n+1}\|^2_{L^2(\Omega)}-
\|\nabla\phi^{n}\|^2_{L^2(\Omega)}+\|\nabla\phi^{n+1}-\nabla\phi^{n}\|^2_{L^2(\Omega)})
+\frac{1}{\varepsilon\tau}(F(\phi^{n+1})-F(\phi^{n}),1)_{\Omega}\\
&+\frac{1}{\tau}(s_1-\frac{1}{2\varepsilon}F''(\eta))\|\phi^{n+1}-\phi^{n}\|^2_{L^2(\Omega)}
+\frac{\delta\kappa}{2\tau}(\|\nabla_{\Gamma}\psi^{n+1}\|^2_{L^2(\Gamma)}-
\|\nabla_{\Gamma}\psi^{n}\|^2_{L^2(\Gamma)}+\|\nabla_{\Gamma}\psi^{n+1}-\nabla_{\Gamma}\psi^{n}\|^2_{L^2(\Gamma)})\\
&+\frac{1}{\delta\tau}(G(\psi^{n+1})-G(\psi^{n}),1)_{\Gamma}+\frac{1}{\tau}
(s_2-\frac{1}{2\delta}G''(\zeta))\|\psi^{n+1}-\psi^{n}\|^2_{L^2(\Gamma)}\\
&=\frac{1}{\tau}[E(\phi^{n+1},\psi^{n+1})-E(\phi^{n},\psi^{n})]+\frac{\varepsilon}{2\tau}
\|\nabla\phi^{n+1}-\nabla\phi^{n}\|^2_{L^2(\Omega)}
+\frac{\delta\kappa}{2\tau}\|\nabla_{\Gamma}\psi^{n+1}-\nabla_{\Gamma}\psi^{n}\|^2_{L^2(\Gamma)}\\
&+\frac{1}{\tau}(s_1-\frac{1}{2\varepsilon}F''(\eta))\|\phi^{n+1}-\phi^{n}\|^2_{L^2(\Omega)}
+\frac{1}{\tau}(s_2-\frac{1}{2\delta}G''(\zeta))\|\psi^{n+1}-\psi^{n}\|^2_{L^2(\Gamma)}.
\end{aligned}
$$
Thus, we have
$$
\begin{aligned}
&\frac{1}{\tau}[E(\phi^{n+1},\psi^{n+1})-E(\phi^{n},\psi^{n})]+\frac{\varepsilon}{2\tau}
\|\nabla\phi^{n+1}-\nabla\phi^{n}\|^2_{L^2(\Omega)}
+\frac{\delta\kappa}{2\tau}\|\nabla_{\Gamma}\psi^{n+1}-\nabla_{\Gamma}\psi^{n}\|^2_{L^2(\Gamma)}\\
&+\frac{1}{\tau}(s_1-\frac{1}{2\varepsilon}F''(\eta))\|\phi^{n+1}-\phi^{n}\|^2_{L^2(\Omega)}
+\frac{1}{\tau}(s_2-\frac{1}{2\delta}G''(\zeta))\|\psi^{n+1}-\psi^{n}\|^2_{L^2(\Gamma)}\\
&=-\|\nabla\mu^{n+1}\|^2_{L^2(\Omega)}-\|\nabla_{\Gamma}\mu^{n+1}_{\Gamma}\|^2_{L^2(\Gamma)}-\frac{1}K\|\mu^{n+1}
-\mu_{\Gamma}^{n+1}\|^2_{L^2(\Gamma)}\leq0.
\end{aligned}
$$
Therefore, under the conditions that
$$s_1\geq\frac{1}{2\varepsilon} \max_{\xi\in\mathbb{R}} F''(\xi)$$
and
$$s_2\geq\frac{1}{2\delta} \max_{\eta\in\mathbb{R}} G''(\eta),$$
we have
$$
\frac{1}{\tau}[E(\phi^{n+1},\psi^{n+1})-E(\phi^{n},\psi^{n})]\leq0,
$$
namely, the scheme (\ref{SIscheme1})-(\ref{SIscheme6}) is energy stable.
\end{proof}

\begin{remark}
The assumption \eqref{s1s2} is reasonable. The energy potential $F$ is a functional with respect to $\phi$ and $\phi$ is a function defined as $\phi: \Omega \rightarrow \mathbb{R}$. Similarly, $G$ is a functional with respect to $\psi$ and $\psi$ is a function defined as $\psi: \Gamma \rightarrow \mathbb{R}$.
And the derivatives in \eqref{s1s2} are with respect to $\phi$ and $\psi$ respectively.
Thus, if the second derivative of $F$ with respect to $\phi$ and the second derivative of $G$ with respect to $\psi$ (namely, $F''$ and $G''$) are bounded, we can choose $s_1$ and $s_2$ large enough to satisfy \eqref{s1s2}.

One example of the energy potentials $F$ and $G$ is the modified double-well potential (also called the truncated double-well potential). Here, the word 'truncated' means that it truncates $\mathbb{R}$ into three parts: $(-\infty, -1)$, $(-1, 1)$ and $(1, \infty)$ and use the quadratic functions to replace the function $\frac{1}4(\phi^2-1)^2$ on $(-\infty, -1)$ and $(1, \infty)$. It reads as follows,
\begin{equation*}
F(\phi)=\left\{\begin{aligned}
&(\phi-1)^2 \qquad \phi>1,\\
&\frac{1}4(\phi^2-1)^2 \quad -1\leq\phi\leq1,\\
&(\phi+1)^2 \qquad \phi<-1.
\end{aligned}
\right.
G(\psi)=\left\{\begin{aligned}
&(\psi-1)^2 \qquad \psi>1,\\
&\frac{1}4(\psi^2-1)^2 \quad -1\leq\psi\leq1,\\
&(\psi+1)^2 \qquad \psi<-1.
\end{aligned}
\right.
\end{equation*}
Obviously, the second derivative of $F$ with respect to $\phi$ and the second derivative of $G$ with respect to $\psi$ are bounded:
$$
\max_{\phi\in\mathbb{R}} |F''(\phi)|=\max_{\psi\in\mathbb{R}} |G''(\psi)|\leq2.
$$
Thus, we can choose $s_1$ and $s_2$ large enough, namely, $s_1>1/\varepsilon$ and $s_2>1/\delta$, so that the assumption \eqref{s1s2} is satisfied.
\end{remark}

\begin{remark}
The proposed scheme (\ref{SIscheme1})-(\ref{SIscheme6}) is first-order in time, linear and unconditionally energy stable, based on the stabilization method. The stabilization method can be directly extended to second-order schemes. However, in that case, the higher-order scheme generally cannot be unconditionally energy stable \cite{shenreview}.
\end{remark}

\section{Error estimates for the stabilized semi-discrete scheme}\label{s4}

In this section, we establish the error estimates for the functions $\phi$ and $\psi$ for the numerical scheme \eqref{SIscheme1}-\eqref{SIscheme6}. Here, the mathematics induction is utilized and the trace theorem is applied to estimate the boundary terms.

Assume that the Lipschitz properties hold for the second derivative of $F$ with respect to $\phi$ and the second derivative of $G$ with respect to $\psi$ (namely, $F''$ and $G''$), and $F''$ and $G''$ are bounded. Precisely, there exists positive constants $L_1$, $L_2$, $K_1$ and $K_2$ that
\begin{equation*}
|F^{''}(\phi_1)-F^{''}(\phi_2)|\leq K_1 |\phi_1 - \phi_2|,
\end{equation*}
\begin{equation}\label{lipassump}
|G^{''}(\psi_1)-G^{''}(\psi_2)|\leq K_2 |\psi_1 - \psi_2|, \quad \mbox{for\ } \phi_1, \phi_2, \psi_1, \psi_2 \in  \mathbb{R},
\end{equation}
\begin{equation}\label{boundassump}
\max_{\phi\in\mathbb{R}} |F^{''}(\phi)|\leq L_1, \quad \max_{\psi\in\mathbb{R}} |G^{''}(\psi)|\leq L_2.
\end{equation}
These assumptions are necessary for error estimates.

\begin{remark}
The assumptions \eqref{lipassump} - \eqref{boundassump} are reasonable. One example of the functionals $F$ and $G$, satisfying the assumptions mentioned above, is the modified double-well potential:
\begin{equation}\label{modifyFG}
F(\phi)=\left\{\begin{aligned}
&(\phi-1)^2 \qquad \phi>1,\\
&\frac{1}4(\phi^2-1)^2 \quad -1\leq\phi\leq1,\\
&(\phi+1)^2 \qquad \phi<-1.
\end{aligned}
\right.
G(\psi)=\left\{\begin{aligned}
&(\psi-1)^2 \qquad \psi>1,\\
&\frac{1}4(\psi^2-1)^2 \quad -1\leq\psi\leq1,\\
&(\psi+1)^2 \qquad \psi<-1.
\end{aligned}
\right.
\end{equation}
Obviously, the Lipschitz property holds for the second derivative of $F$ with respect to $\phi$ and the second derivative of $G$ with respect to $\psi$:
\begin{equation*}
|F^{''}(\phi_1)-F^{''}(\phi_2)|\leq 6 |\phi_1 - \phi_2|,
\end{equation*}
\begin{equation*}
|G^{''}(\psi_1)-G^{''}(\psi_2)|\leq 6 |\psi_1 - \psi_2|, \quad \mbox{for\ } \phi_1, \phi_2, \psi_1, \psi_2 \in  \mathbb{R},
\end{equation*}
and
\begin{equation*}
\max_{\phi\in\mathbb{R}} |F''(\phi)|=\max_{\psi\in\mathbb{R}} |G''(\psi)|\leq2.
\end{equation*}

\end{remark}

The PDE system \eqref{CHK} can be rewritten as the following truncated form,
\begin{eqnarray}\label{truncatePDE}
&&\frac{\phi(t^{n+1})-\phi(t^{n})}{\tau}=\Delta\mu(t^{n+1})+R_{\phi}^{n+1}, \quad \mbox{in}\  \Omega,\\
\label{truncatePDE2}
&&\mu(t^{n+1})=-\varepsilon\Delta\phi(t^{n+1})+\frac{1}{\varepsilon}F'(\phi(t^{n}))
+s_1(\phi(t^{n+1})-\phi(t^{n}))+R_{\mu}^{n+1}, \quad \mbox{in}\ \Omega,\\
\label{truncatePDE3}
&&K\partial_{\mathbf{n}}\mu(t^{n+1})=\mu_{\Gamma}(t^{n+1})-\mu(t^{n+1}) \quad \mbox{on}\  \Gamma,\\
\label{truncatePDE4}
&&\phi(t^{n+1})|_{\Gamma}=\psi(t^{n+1}), \quad \mbox{on}\  \Gamma,\\
\label{truncatePDE5}
&&\frac{\psi(t^{n+1})-\psi(t^{n})}{\tau}=\Delta_{\Gamma}\mu_{\Gamma}(t^{n+1})-\partial_{\mathbf{n}}\mu(t^{n+1})+R_{\psi}^{n+1}, \quad \mbox{on}\  \Gamma,\\
\label{truncatePDE6}
&&\mu_{\Gamma}(t^{n+1})=-\delta\kappa\Delta_{\Gamma}\psi(t^{n+1})+\frac{1}{\delta}G'(\psi(t^{n}))
+\varepsilon\partial_{\mathbf{n}}\phi(t^{n+1}) \nonumber\\
&&\qquad \qquad +s_2(\psi(t^{n+1})-\psi(t^{n}))+R_{\Gamma}^{n+1}, \quad \mbox{on}\  \Gamma,
\end{eqnarray}
where
\begin{equation}
R_{\phi}^{n+1}=\frac{\phi(t^{n+1})-\phi(t^{n})}{\tau}-\phi_t(t^{n+1}),
\end{equation}
\begin{equation}
R_{\psi}^{n+1}=\frac{\psi(t^{n+1})-\psi(t^{n})}{\tau}-\psi_t(t^{n+1}),
\end{equation}
\begin{equation}
R_{\mu}^{n+1}=\frac{1}{\varepsilon}F'(\phi(t^{n+1}))-\frac{1}{\varepsilon}F'(\phi(t^{n}))
-s_1(\phi(t^{n+1})-\phi(t^{n})),
\end{equation}
\begin{equation}
R_{\Gamma}^{n+1}=\frac{1}{\delta}G'(\psi(t^{n+1}))-\frac{1}{\delta}G'(\psi(t^{n}))
-s_2(\psi(t^{n+1})-\psi(t^{n})).
\end{equation}

We assume that the exact solution $(\phi,\psi, \mu, \mu_{\Gamma})$ of the system \eqref{CHK} is sufficiently smooth, or possesses the following regularity:
\begin{equation}
(A_1):
\begin{aligned}
&\phi,\phi_t,\phi_{tt}\in L^{\infty}(0,T; H^{m_1}(\Omega));\\
&\mu \in L^{\infty}(0,T; H^{m_2}(\Omega));\\
&\mu_{\Gamma} \in L^{\infty}(0,T; H^{m_3}(\Gamma));\\
\end{aligned}
\end{equation}
with $m_1, m_2, m_3$ sufficiently large (the assumption that $m_1\geqslant7/2$, $m_2\geqslant3/2$ and $m_3\geqslant1$ is suitable for the following error analysis). Due to the trace theorem and the linearity of the trace operator, the trace $\psi$ possesses the regularity:
\begin{equation}
(A_2):
\begin{aligned}
&\psi,\psi_t,\psi_{tt}\in L^{\infty}(0,T; H^{m_1-1/2}(\Gamma))
\end{aligned}
\end{equation}

%
%
From the Taylor expansion, it's easy to prove that

\begin{lemma}
The truncation errors satisfy
\begin{equation}
\begin{aligned}
&\|R_{\phi,\tau}\|_{l^{\infty}(H^1(\Omega))}+\|R_{\mu,\tau}\|_{l^{\infty}(H^1(\Omega))}\lesssim \tau,\\
&\|R_{\psi,\tau}\|_{l^{\infty}(H^1(\Gamma))}+\|R_{\Gamma,\tau}\|_{l^{\infty}(H^1(\Gamma))}\lesssim \tau.\\
\end{aligned}
\end{equation}
\end{lemma}

By subtracting \eqref{truncatePDE}-\eqref{truncatePDE6} from the corresponding scheme \eqref{SIscheme1}-\eqref{SIscheme6}, we derive the error equations as follows,
\begin{eqnarray}\label{erroreq}
&&\frac{1}{\tau}(e_{\phi}^{n+1}-e_{\phi}^{n})=\Delta e_{\mu}^{n+1}+R_{\phi}^{n+1}, \quad \mbox{in}\  \Omega,\\
\label{erroreq2}
&&e_{\mu}^{n+1}=-\varepsilon\Delta e_{\phi}^{n+1}+\frac{1}{\varepsilon}(F'(\phi(t^{n}))-F'(\phi^{n}))+s_1(e_{\phi}^{n+1}-e_{\phi}^{n})
+R_{\mu}^{n+1}, \quad \mbox{in}\ \Omega,\\
\label{erroreq3}
&&K\partial_{\mathbf{n}}e_{\mu}^{n+1}=e_{\Gamma}^{n+1}-e_{\mu}^{n+1}, \quad \mbox{on}\  \Gamma,\\
\label{erroreq4}
&&e_{\phi}^{n+1}|_{\Gamma}=e_{\psi}^{n+1}, \quad \mbox{on}\  \Gamma,\\
\label{erroreq5}
&&\frac{1}{\tau}(e_{\psi}^{n+1}-e_{\psi}^{n})=\Delta_{\Gamma}e_{\Gamma}^{n+1}-\partial_{\mathbf{n}}e_{\mu}^{n+1}
+R_{\psi}^{n+1}, \quad \mbox{on}\  \Gamma,\\
\label{erroreq6}
&&e_{\Gamma}^{n+1}=-\delta\kappa\Delta_{\Gamma}e_{\psi}^{n+1}+\frac{1}{\delta}(G'(\psi(t^{n}))-G'(\psi^{n}))
+\varepsilon\partial_{\mathbf{n}}e_{\phi}^{n+1} \nonumber\\
&&\qquad \quad +s_2(e_{\psi}^{n+1}-e_{\psi}^{n})+R_{\Gamma}^{n+1}, \quad \mbox{on}\  \Gamma.
\end{eqnarray}

Here, the error functions are defined as
\begin{equation}
\begin{aligned}
&e_{\phi}^n=\phi(t^{n})-\phi^n,\qquad e_{\mu}^n=\mu(t^{n})-\mu^n,\\
&e_{\psi}^n=\psi(t^{n})-\psi^n, \qquad e_{\Gamma}^n=\mu_{\Gamma}(t^{n})-\mu_{\Gamma}^n. 
\end{aligned}
\end{equation}
Obviously, we have $e_{\phi}^n|_{\Gamma}=e_{\psi}^n$. The corresponding sequence of error functions are denoted as $e_{\phi,\tau}$, $e_{\psi,\tau}$, $e_{\mu,\tau}$ and $e_{\Gamma,\tau}$.

Thus we can establish the estimates for the scheme \eqref{SIscheme1}-\eqref{SIscheme6} as follows.

\begin{theorem}
Provided that the exact solutions are sufficiently smooth, there exists some $\tau_0>0$ such that when $\tau < \tau_0$,
the solution $(\phi^m,\psi^m)$ ($0\leq m\leq \bigg{[}\frac{T}{\tau}\bigg{]}-1$) of the scheme \eqref{SIscheme1}-\eqref{SIscheme6}  satisfy the following error estimate
\begin{equation}\label{finalestimate}
\begin{aligned}
&\|e_{\phi,\tau}\|_{l^{\infty}(H^1(\Omega))}+\|e_{\psi,\tau}\|_{l^{\infty}(H^1(\Gamma))}\lesssim\tau.
\end{aligned}
\end{equation}
Here, the error functions are defined as
\begin{equation}
\begin{aligned}
&e_{\phi}^n=\phi(t^{n})-\phi^n,\qquad e_{\mu}^n=\mu(t^{n})-\mu^n,\\
&e_{\psi}^n=\psi(t^{n})-\psi^n, \qquad e_{\Gamma}^n=\mu_{\Gamma}(t^{n})-\mu_{\Gamma}^n,\\ &e_{\phi}^n|_{\Gamma}=e_{\psi}^n.
\end{aligned}
\end{equation}
The corresponding sequence of error functions are denoted as $e_{\phi,\tau}$, $e_{\psi,\tau}$, $e_{\mu,\tau}$ and $e_{\Gamma,\tau}$, and the discrete norm $\|\cdot\|_{l^{\infty}(\cdot)}$ is defined as Eq. \eqref{disnorm}.

\end{theorem}

\begin{proof}
We use the mathematical induction to prove this theorem.
When $m=0$, we have $e_{\phi}^0=e_{\psi}^0=\nabla e_{\phi}^0=\nabla_{\Gamma} e_{\psi}^0=0$. Obviously, \eqref{finalestimate} holds.
Assuming that \eqref{finalestimate} holds for all $n\leq m$, we need to show that \eqref{finalestimate} holds for $e_{\phi}^{m+1}$ and $e_{\psi}^{m+1}$.

For each $n\leq m$, by taking the $L^2$ inner product of \eqref{erroreq} with $\tau e_{\mu}^{n+1}$ in $\Omega$, we obtain
$$
(e_{\phi}^{n+1}-e_{\phi}^{n}, e_{\mu}^{n+1})_{\Omega}+\tau\|\nabla e_{\mu}^{n+1}\|_{\Omega}^2=\tau(\partial_{\mathbf{n}}e_{\mu}^{n+1}, e_{\mu}^{n+1})_{\Gamma}+\tau(R_{\phi}^{n+1},e_{\mu}^{n+1})_{\Omega}.
$$
By taking the $L^2$ inner product of \eqref{erroreq} with $\varepsilon\tau e_{\phi}^{n+1}$ in $\Omega$, we obtain
$$
\begin{aligned}
\frac{\varepsilon}2(\| e_{\phi}^{n+1}\|_{\Omega}^2-\|e_{\phi}^{n}\|_{\Omega}^2+\| e_{\phi}^{n+1}-e_{\phi}^{n}\|_{\Omega}^2)&=-\varepsilon\tau(\nabla e_{\mu}^{n+1},\nabla e_{\phi}^{n+1})_{\Omega}\\
&+\varepsilon\tau (\partial_{\mathbf{n}}e_{\mu}^{n+1}, e_{\phi}^{n+1})_{\Gamma}+\varepsilon\tau(R_{\phi}^{n+1},e_{\phi}^{n+1})_{\Omega}.
\end{aligned}
$$
By taking the $L^2$ inner product of \eqref{erroreq2} with $-(e_{\phi}^{n+1}-e_{\phi}^{n})$ in $\Omega$, we obtain
$$
\begin{aligned}
&-(e_{\mu}^{n+1}, e_{\phi}^{n+1}-e_{\phi}^{n})_{\Omega}+\frac{\varepsilon}2(\| \nabla e_{\phi}^{n+1}\|_{\Omega}^2-\|\nabla e_{\phi}^{n}\|_{\Omega}^2+\| \nabla e_{\phi}^{n+1}-\nabla e_{\phi}^{n}\|_{\Omega}^2)+s_1\|e_{\phi}^{n+1}-e_{\phi}^{n}\|_{\Omega}^2=\\
&\varepsilon(\partial_{\mathbf{n}}e_{\phi}^{n+1},
e_{\phi}^{n+1}-e_{\phi}^{n})_{\Gamma}-\frac{1}{\varepsilon}(F'(\phi(t^{n}))-F'(\phi^{n}), e_{\phi}^{n+1}-e_{\phi}^{n})_{\Omega}-(R_{\mu}^{n+1},e_{\phi}^{n+1}-e_{\phi}^{n})_{\Omega}.
\end{aligned}
$$


By combining the equations above, we derive
 \begin{equation}\label{errorin}
\begin{aligned}
&\frac{\varepsilon}2(\| \nabla e_{\phi}^{n+1}\|_{\Omega}^2-\|\nabla e_{\phi}^{n}\|_{\Omega}^2+\| \nabla e_{\phi}^{n+1}-\nabla e_{\phi}^{n}\|_{\Omega}^2)+s_1\|e_{\phi}^{n+1}-e_{\phi}^{n}\|_{\Omega}^2\\
&+\frac{\varepsilon}2(\| e_{\phi}^{n+1}\|_{\Omega}^2-\| e_{\phi}^{n}\|_{\Omega}^2+\| e_{\phi}^{n+1}- e_{\phi}^{n}\|_{\Omega}^2)+\tau\|\nabla e_{\mu}^{n+1}\|_{\Omega}^2\\
&=\tau(\partial_{\mathbf{n}}e_{\mu}^{n+1}, e_{\mu}^{n+1})_{\Gamma}+  \tau(R_{\phi}^{n+1}, e_{\mu}^{n+1})_{\Omega}-\varepsilon\tau(\nabla e_{\mu}^{n+1}, \nabla e_{\phi}^{n+1})_{\Omega}\\
&+\varepsilon\tau (\partial_{\mathbf{n}}e_{\mu}^{n+1}, e_{\phi}^{n+1})_{\Gamma}+\varepsilon\tau(R_{\phi}^{n+1},e_{\phi}^{n+1})_{\Omega}+\varepsilon(\partial_{\mathbf{n}}e_{\phi}^{n+1},
e_{\phi}^{n+1}-e_{\phi}^{n})_{\Gamma}\\
&-\frac{1}{\varepsilon}(F'(\phi(t^{n}))-F'(\phi^{n}), e_{\phi}^{n+1}-e_{\phi}^{n})_{\Omega}-(R_{\mu}^{n+1}, e_{\phi}^{n+1}-e_{\phi}^{n})_{\Omega}.
\end{aligned}
\end{equation}

For the boundary term, by taking the $L^2$ inner product of \eqref{erroreq5} with $\tau e_{\Gamma}^{n+1}$ on $\Gamma$, we obtain
$$
(e_{\psi}^{n+1}-e_{\psi}^{n}, e_{\Gamma}^{n+1})_{\Gamma}+\tau\|\nabla_{\Gamma} e_{\Gamma}^{n+1}\|_{\Gamma}^2+\tau(\partial_{\mathbf{n}}e_{\mu}^{n+1}, e_{\Gamma}^{n+1})_{\Gamma}=\tau(R_{\psi}^{n+1},e_{\Gamma}^{n+1})_{\Gamma}.
$$
By taking the $L^2$ inner product of \eqref{erroreq5} with $\varepsilon\tau e_{\psi}^{n+1}$ on $\Gamma$, we obtain
$$
\begin{aligned}
\frac{\varepsilon}2(\| e_{\psi}^{n+1}\|_{\Gamma}^2-\|e_{\psi}^{n}\|_{\Gamma}^2+\| e_{\psi}^{n+1}-e_{\psi}^{n}\|_{\Gamma}^2)&=-\varepsilon\tau(\nabla_{\Gamma} e_{\Gamma}^{n+1},\nabla_{\Gamma} e_{\psi}^{n+1})_{\Gamma}\\
&-\varepsilon\tau(\partial_{\mathbf{n}}e_{\mu}^{n+1}, e_{\psi}^{n+1})_{\Gamma}
+\varepsilon\tau(R_{\psi}^{n+1},e_{\psi}^{n+1})_{\Gamma},
\end{aligned}
$$
where the boundary terms vanish due to $\Gamma$ is closed.
By taking the $L^2$ inner product of \eqref{erroreq6} with $-(e_{\psi}^{n+1}-e_{\psi}^{n})$ on $\Gamma$, we obtain
$$
\begin{aligned}
&-(e_{\Gamma}^{n+1}, e_{\psi}^{n+1}-e_{\psi}^{n})_{\Gamma}+\frac{\delta\kappa}2(\| \nabla_{\Gamma} e_{\psi}^{n+1}\|_{\Gamma}^2-\|\nabla_{\Gamma} e_{\psi}^{n}\|_{\Gamma}^2+\| \nabla_{\Gamma} e_{\psi}^{n+1}-\nabla_{\Gamma} e_{\psi}^{n}\|_{\Gamma}^2)
+s_2\|e_{\psi}^{n+1}-e_{\psi}^{n}\|_{\Gamma}^2\\
&=-\varepsilon(\partial_{\mathbf{n}}e_{\phi}^{n+1},
e_{\psi}^{n+1}-e_{\psi}^{n})_{\Gamma}-\frac{1}{\delta}(G'(\psi(t^{n}))-G'(\psi^{n}), e_{\psi}^{n+1}-e_{\psi}^{n})_{\Gamma}-(R_{\Gamma}^{n+1},e_{\psi}^{n+1}-e_{\psi}^{n})_{\Gamma}.
\end{aligned}
$$

%
By combining the equations above, we derive
 \begin{equation}\label{errorboun}
\begin{aligned}
&\frac{\delta\kappa}2(\| \nabla_{\Gamma} e_{\psi}^{n+1}\|_{\Gamma}^2-\|\nabla_{\Gamma} e_{\psi}^{n}\|_{\Gamma}^2+\| \nabla_{\Gamma} e_{\psi}^{n+1}-\nabla_{\Gamma} e_{\psi}^{n}\|_{\Gamma}^2)+s_2\|e_{\psi}^{n+1}-e_{\psi}^{n}\|_{\Gamma}^2\\
&+\frac{\varepsilon}2(\| e_{\psi}^{n+1}\|_{\Gamma}^2-\| e_{\psi}^{n}\|_{\Gamma}^2+\| e_{\psi}^{n+1}- e_{\psi}^{n}\|_{\Gamma}^2)+\tau\|\nabla_{\Gamma} e_{\Gamma}^{n+1}\|_{\Gamma}^2+\tau(\partial_{\mathbf{n}}e_{\mu}^{n+1}, e_{\Gamma}^{n+1})_{\Gamma}\\
&=\tau(R_{\psi}^{n+1},e_{\Gamma}^{n+1})_{\Gamma}-\varepsilon\tau(\nabla_{\Gamma} e_{\Gamma}^{n+1},\nabla_{\Gamma} e_{\psi}^{n+1})_{\Gamma}-\varepsilon\tau( \partial_{\mathbf{n}}e_{\mu}^{n+1}, e_{\psi}^{n+1})_{\Gamma}\\
&+\tau\varepsilon(R_{\psi}^{n+1},e_{\psi}^{n+1})_{\Gamma}
-\frac{1}{\delta}(G'(\psi(t^{n}))-G'(\psi^{n}), e_{\psi}^{n+1}-e_{\psi}^{n})_{\Gamma}\\
&-\varepsilon(\partial_{\mathbf{n}}e_{\phi}^{n+1},
e_{\psi}^{n+1}-e_{\psi}^{n})_{\Gamma}-(R_{\Gamma}^{n+1},e_{\psi}^{n+1}-e_{\psi}^{n})_{\Gamma}.
\end{aligned}
\end{equation}

By combining \eqref{errorin} and \eqref{errorboun} together, we derive
\begin{equation}\label{errorzong}
\begin{aligned}
&\frac{\varepsilon}2(\| \nabla e_{\phi}^{n+1}\|_{\Omega}^2-\|\nabla e_{\phi}^{n}\|_{\Omega}^2+\| \nabla e_{\phi}^{n+1}-\nabla e_{\phi}^{n}\|_{\Omega}^2)+\frac{\varepsilon}2(\| e_{\phi}^{n+1}\|_{\Omega}^2-\| e_{\phi}^{n}\|_{\Omega}^2+\| e_{\phi}^{n+1}- e_{\phi}^{n}\|_{\Omega}^2)\\
&+\frac{\delta\kappa}2(\| \nabla_{\Gamma} e_{\psi}^{n+1}\|_{\Gamma}^2-\|\nabla_{\Gamma} e_{\psi}^{n}\|_{\Gamma}^2+\| \nabla_{\Gamma} e_{\psi}^{n+1}-\nabla_{\Gamma} e_{\psi}^{n}\|_{\Gamma}^2)\\
&+\frac{\varepsilon}2(\| e_{\psi}^{n+1}\|_{\Gamma}^2-\| e_{\psi}^{n}\|_{\Gamma}^2+\| e_{\psi}^{n+1}- e_{\psi}^{n}\|_{\Gamma}^2)
+s_1\|e_{\phi}^{n+1}-e_{\phi}^{n}\|_{\Omega}^2+s_2\|e_{\psi}^{n+1}-e_{\psi}^{n}\|_{\Gamma}^2\\
&+\tau\|\nabla e_{\mu}^{n+1}\|_{\Omega}^2+\tau\|\nabla_{\Gamma} e_{\Gamma}^{n+1}\|_{\Gamma}^2+K\tau\|\partial_{\mathbf{n}}e_{\mu}^{n+1}\|_{\Gamma}^2\\
&=\varepsilon\tau(R_{\phi}^{n+1},e_{\phi}^{n+1})_{\Omega}+\tau\varepsilon(R_{\psi}^{n+1},e_{\psi}^{n+1})_{\Gamma}
\quad (:=\mbox{term }A_1)\\
&+\tau(R_{\phi}^{n+1},e_{\mu}^{n+1})_{\Omega}+\tau(R_{\psi}^{n+1},e_{\Gamma}^{n+1})_{\Gamma}\quad (:=\mbox{term }A_2)\\
&-\varepsilon\tau(\nabla e_{\mu}^{n+1},\nabla e_{\phi}^{n+1})_{\Omega}-\varepsilon\tau(\nabla_{\Gamma} e_{\Gamma}^{n+1},\nabla_{\Gamma} e_{\psi}^{n+1})_{\Gamma} \quad (:=\mbox{term }A_3)\\
&-\frac{1}{\varepsilon}(F'(\phi(t^{n}))-F'(\phi^{n}), e_{\phi}^{n+1}-e_{\phi}^{n})_{\Omega}-(R_{\mu}^{n+1},e_{\phi}^{n+1}-e_{\phi}^{n})_{\Omega}
\quad (:=\mbox{term }A_4)\\
&-\frac{1}{\delta}(G'(\psi(t^{n}))-G'(\psi^{n}), e_{\psi}^{n+1}-e_{\psi}^{n})_{\Gamma}-(R_{\Gamma}^{n+1},e_{\psi}^{n+1}-e_{\psi}^{n})_{\Gamma}
\quad (:=\mbox{term }A_5).
\end{aligned}
\end{equation}

For the term $A_1$, we have
\begin{equation}\label{A1}
\begin{aligned}
&\varepsilon\tau(R_{\phi}^{n+1},e_{\phi}^{n+1})_{\Omega}+\tau\varepsilon(R_{\psi}^{n+1},e_{\psi}^{n+1})_{\Gamma}\\
&\leq \varepsilon\tau\|R_{\phi}^{n+1}\|_{\Omega}\|e_{\phi}^{n+1}\|_{\Omega}+
\varepsilon\tau\|R_{\psi}^{n+1}\|_{\Gamma}\|e_{\psi}^{n+1}\|_{\Gamma}\\
& \leq \frac{\varepsilon\tau}{2} \|e_{\phi}^{n+1}\|_{\Omega}^2
+\frac{\varepsilon\tau}{2} \|e_{\psi}^{n+1}\|_{\Gamma}^2+C_1\varepsilon \tau^3,
\end{aligned}
\end{equation}
where $C_1$ is a constant independent of $\tau$ and $\varepsilon$. Here, we use the estimates for the truncation terms $R_{\phi}^{n+1}$ and $R_{\psi}^{n+1}$.

In this section, we define $H^n=F'(\phi(t^{n}))-F'(\phi^{n})$ for simplicity. It can be rewritten as
\begin{equation}
H^n=e_{\phi}^{n}\int_0^1 F''(s\phi(t^{n})+(1-s)\phi^{n}) ds.
\end{equation}
Then we have $\|H^n\|_{\Omega}\lesssim\|e_{\phi}^{n}\|_{\Omega}$ since $F''$ is bounded.
By taking the gradient of $H^n$, we have
\begin{equation}
\begin{aligned}
\nabla H^n&=F''(\phi(t^{n}))\nabla \phi(t^{n})-F''(\phi^{n})\nabla\phi^{n}=(F''(\phi(t^{n}))-F''(\phi^{n}))\nabla \phi(t^{n})+F''(\phi^{n})\nabla e_{\phi}^{n}.
\end{aligned}
\end{equation}
Since $F''$ is bounded and Lipschitz and $\phi\in L^{\infty}(0,T; H^{m_1}(\Omega))$ with $m_1$ sufficiently large, we have
\begin{equation}
\|\nabla H^n\|_{\Omega}\lesssim\|e_{\phi}^{n}\|_{\Omega}+\|\nabla e_{\phi}^{n}\|_{\Omega}.
\end{equation}
Similarly, we define $\tilde{H}^n=G'(\psi(t^{n}))-G'(\psi^{n})$. Since $G''$ is bounded and Lipschitz and $\psi\in L^{\infty}(0,T; H^{m_1-1/2}(\Gamma))$ with $m_1$ sufficiently large, we have
\begin{equation}
\begin{aligned}
&\|\tilde{H}^n\|_{\Gamma}\lesssim\|e_{\psi}^{n}\|_{\Gamma},\\
&\|\nabla_{\Gamma} \tilde{H}^n\|_{\Gamma}\lesssim\|e_{\psi}^{n}\|_{\Gamma}+\|\nabla_{\Gamma} e_{\psi}^{n}\|_{\Gamma}.
\end{aligned}
\end{equation}

For the term $A_2$, we have
\begin{equation}\label{A2}
\begin{aligned}
&\tau(R_{\phi}^{n+1},e_{\mu}^{n+1})_{\Omega}+\tau(R_{\psi}^{n+1},e_{\Gamma}^{n+1})_{\Gamma}\\
&=\tau(R_{\phi}^{n+1}, -\varepsilon\Delta e_{\phi}^{n+1}+\frac{1}{\varepsilon}H^n
+s_1(e_{\phi}^{n+1}-e_{\phi}^{n})+R_{\mu}^{n+1})_{\Omega}\\
&+\tau(R_{\psi}^{n+1}, -\delta\kappa\Delta_{\Gamma}e_{\psi}^{n+1}+\frac{1}{\delta}\tilde{H}^n
+\varepsilon\partial_{\mathbf{n}}e_{\phi}^{n+1}
+s_2(e_{\psi}^{n+1}-e_{\psi}^{n})+R_{\Gamma}^{n+1})_{\Gamma}\\
&=\varepsilon\tau(\nabla R_{\phi}^{n+1}, \nabla e_{\phi}^{n+1})_{\Omega}-\varepsilon\tau(\partial_{\mathbf{n}}e_{\phi}^{n+1}, R_{\phi}^{n+1})_{\Gamma}+\frac{\tau}{\varepsilon}(H^n, R_{\phi}^{n+1})_{\Omega}+
s_1\tau(R_{\phi}^{n+1}, e_{\phi}^{n+1}-e_{\phi}^{n})_{\Omega}\\
&+\tau(R_{\phi}^{n+1}, R_{\mu}^{n+1})_{\Omega}+\tau\delta\kappa(\nabla_{\Gamma}R_{\psi}^{n+1}, \nabla_{\Gamma} e_{\psi}^{n+1})_{\Gamma}+\frac{\tau}{\delta}(\tilde{H}^n, R_{\psi}^{n+1})_{\Gamma}+\varepsilon\tau(\partial_{\mathbf{n}}e_{\phi}^{n+1}, R_{\psi}^{n+1})_{\Gamma}\\
&+s_2\tau(R_{\psi}^{n+1}, e_{\psi}^{n+1}-e_{\psi}^{n})_{\Gamma}+\tau(R_{\psi}^{n+1},R_{\Gamma}^{n+1})_{\Gamma}\\
&\leq \varepsilon\tau \|\nabla R_{\phi}^{n+1}\|_{\Omega}\|\nabla e_{\phi}^{n+1}\|_{\Omega}+\frac{\tau}{\varepsilon}\|H^n\|_{\Omega}\| R_{\phi}^{n+1}\|_{\Omega}+s_1\tau\|R_{\phi}^{n+1}\|_{\Omega}\| e_{\phi}^{n+1}-e_{\phi}^{n}\|_{\Omega}\\
&+\tau\|R_{\phi}^{n+1}\|_{\Omega}\|R_{\mu}^{n+1}\|_{\Omega}+\tau\delta\kappa\|
\nabla_{\Gamma}R_{\psi}^{n+1}\|_{\Gamma}\|\nabla_{\Gamma} e_{\psi}^{n+1}\|_{\Gamma}+
\frac{\tau}{\delta}\|\tilde{H}^n\|_{\Gamma}\|R_{\psi}^{n+1}\|_{\Gamma}\\
&+s_2\tau\|R_{\psi}^{n+1}\|_{\Gamma}\|e_{\psi}^{n+1}-e_{\psi}^{n}\|_{\Gamma}+
\tau\|R_{\psi}^{n+1}\|_{\Gamma}\|R_{\Gamma}^{n+1}\|_{\Gamma}\\
&\leq C_2 \tau^3+\frac{\varepsilon\tau}2\|\nabla e_{\phi}^{n+1}\|_{\Omega}^2+C_3 \tau \| e_{\phi}^{n}\|_{\Omega}^2+\frac{s_1\tau}{2}\| e_{\phi}^{n+1}-e_{\phi}^{n}\|_{\Omega}^2\\
&+\frac{\tau\delta\kappa}{2}\|\nabla_{\Gamma} e_{\psi}^{n+1}\|_{\Gamma}^2+C_4 \tau\| e_{\psi}^{n}\|_{\Gamma}^2+\frac{s_2\tau}{2}\| e_{\psi}^{n+1}-e_{\psi}^{n}\|_{\Gamma}^2,
\end{aligned}
\end{equation}
where $C_i$ ($i=2,3,4$) are constants independent of $\tau$ and we use the estimates for $H^n$ and $\tilde{H}^n$ and the truncation terms $R_{\phi}^{n+1}$, $R_{\psi}^{n+1}$, $R_{\mu}^{n+1}$ and $R_{\Gamma}^{n+1}$. The fact that $R_{\phi}^{n+1}|_{\Gamma}=\gamma(R_{\phi}^{n+1})=R_{\psi}^{n+1}$ is also applied, where $\gamma$ is the trace operator.

We estimate $A_3$ as follows
\begin{equation}\label{A3}
\begin{aligned}
&-\varepsilon\tau(\nabla e_{\mu}^{n+1},\nabla e_{\phi}^{n+1})_{\Omega}-\varepsilon\tau(\nabla_{\Gamma} e_{\Gamma}^{n+1},\nabla_{\Gamma} e_{\psi}^{n+1})_{\Gamma}\\
&\leq 2\varepsilon^2\tau\|\nabla e_{\phi}^{n+1}\|_{\Omega}^2+\frac{\tau}8 \|\nabla e_{\mu}^{n+1}\|_{\Omega}^2+2\varepsilon^2\tau\|\nabla_{\Gamma} e_{\psi}^{n+1}\|_{\Gamma}^2+\frac{\tau}8\|\nabla_{\Gamma} e_{\Gamma}^{n+1}\|_{\Gamma}^2.
\end{aligned}
\end{equation}

For the first term in $A_4$, we have
\begin{equation}\label{A4_1_1}
\begin{aligned}
&-\frac{1}{\varepsilon}(F'(\phi(t^{n}))-F'(\phi^{n}), e_{\phi}^{n+1}-e_{\phi}^{n})_{\Omega}\\
&=-\frac{\tau}{\varepsilon}(H^n, \frac{e_{\phi}^{n+1}-e_{\phi}^{n}}{\tau})_{\Omega}
=-\frac{\tau}{\varepsilon}(H^n, \Delta e_{\mu}^{n+1}+R_{\phi}^{n+1})_{\Omega}\\
&=\frac{\tau}{\varepsilon}(\nabla H^n, \nabla e_{\mu}^{n+1})_{\Omega}-\frac{\tau}{\varepsilon}(H^n, \partial_{\mathbf{n}}e_{\mu}^{n+1})_{\Gamma} -\frac{\tau}{\varepsilon}(H^n, R_{\phi}^{n+1})_{\Omega}\\
&\leq\frac{\tau}{\varepsilon}\|\nabla H^n\|_{\Omega}\|\nabla e_{\mu}^{n+1}\|_{\Omega}
+\frac{\tau}{\varepsilon}\|H^n\|_{\Gamma}\|\partial_{\mathbf{n}}e_{\mu}^{n+1}\|_{\Gamma}+
\frac{\tau}{\varepsilon}\|H^n\|_{\Omega}\|R_{\phi}^{n+1}\|_{\Omega}.
\end{aligned}
\end{equation}
Applying the trace theorem,
$$
\|H^n\|_{\Gamma}=\|\gamma H^n\|_{\Gamma}\lesssim \|H^n\|_{H^1(\Omega)}\lesssim\|H^n\|_{\Omega}+\|\nabla H^n\|_{\Omega}\lesssim \|e_{\phi}^{n}\|_{\Omega}+\|\nabla e_{\phi}^{n}\|_{\Omega}\lesssim \tau,
$$
where we use the assumption that $e_{\phi}^{n}$ satisfies the error estimate \eqref{finalestimate},
we obtain
\begin{equation}\label{A4_1_2}
\begin{aligned}
&-\frac{1}{\varepsilon}(H^n, e_{\phi}^{n+1}-e_{\phi}^{n})_{\Omega}\\
&\leq C_5 \tau (\|e_{\phi}^{n}\|_{\Omega}+\|\nabla e_{\phi}^{n}\|_{\Omega})\|\nabla e_{\mu}^{n+1}\|_{\Omega}+C_6 \tau(\|e_{\phi}^{n}\|_{\Omega}+\|\nabla e_{\phi}^{n}\|_{\Omega})\|\partial_{\mathbf{n}}e_{\mu}^{n+1}\|_{\Gamma}\\
&+C_7 \tau \|e_{\phi}^{n}\|_{\Omega}\|R_{\phi}^{n+1}\|_{\Omega}\\
&\leq C_8 \tau^3+\frac{\tau}{4}\|\nabla e_{\mu}^{n+1}\|_{\Omega}^2+\frac{K\tau}{16}\|\partial_{\mathbf{n}}e_{\mu}^{n+1}\|_{\Gamma}^2.
\end{aligned}
\end{equation}
Here, $C_i$ ($i=5,6,7,8$) are constants independent of $\tau$ and we use the estimates for $H^n$ and $R_{\phi}^{n+1}$.

For the second term in $A_4$, we have
\begin{equation}\label{A4_2_1}
\begin{aligned}
&-(R_{\mu}^{n+1},e_{\phi}^{n+1}-e_{\phi}^{n})_{\Omega}
=-\tau(R_{\mu}^{n+1},\frac{e_{\phi}^{n+1}-e_{\phi}^{n}}{\tau})_{\Omega}\\
&=-\tau(R_{\mu}^{n+1}, \Delta e_{\mu}^{n+1}+R_{\phi}^{n+1})_{\Omega}\\
&=\tau(\nabla R_{\mu}^{n+1}, \nabla e_{\mu}^{n+1})_{\Omega}-\tau (R_{\mu}^{n+1}, \partial_{\mathbf{n}}e_{\mu}^{n+1})_{\Gamma}-\tau(R_{\mu}^{n+1},R_{\phi}^{n+1})_{\Omega}\\
&\leq \tau \|\nabla R_{\mu}^{n+1}\|_{\Omega}\|\nabla e_{\mu}^{n+1}\|_{\Omega}+\tau\|R_{\mu}^{n+1}\|_{\Gamma}\|\partial_{\mathbf{n}}e_{\mu}^{n+1}\|_{\Gamma}+\tau \|R_{\mu}^{n+1}\|_{\Omega}\|R_{\phi}^{n+1}\|_{\Omega}\\
&\leq 2\tau\|\nabla R_{\mu}^{n+1}\|_{\Omega}^2+\frac{\tau}8\|\nabla e_{\mu}^{n+1}\|_{\Omega}^2+\frac{8\tau}{K}\|R_{\mu}^{n+1}\|_{\Gamma}^2+\frac{K\tau}{32}
\|\partial_{\mathbf{n}}e_{\mu}^{n+1}\|_{\Gamma}^2\\
&+\frac{\tau}2\| R_{\mu}^{n+1}\|_{\Omega}^2+\frac{\tau}2\| R_{\phi}^{n+1}\|_{\Omega}^2\\
&\leq C_9 \tau^3+\frac{\tau}8\|\nabla e_{\mu}^{n+1}\|_{\Omega}^2+\frac{K\tau}{32}
\|\partial_{\mathbf{n}}e_{\mu}^{n+1}\|_{\Gamma}^2,
\end{aligned}
\end{equation}
where $C_9$ is a constant independent of $\tau$. And here, we apply the trace theorem that
$$
\|R_{\mu}^{n+1}\|_{\Gamma}=\|\gamma R_{\mu}^{n+1}\|_{\Gamma}\lesssim \|R_{\mu}^{n+1}\|_{H^1(\Omega)}\lesssim\|R_{\mu}^{n+1}\|_{\Omega}+\|\nabla R_{\mu}^{n+1}\|_{\Omega}\lesssim \tau,
$$
and use the estimates for $R_{\mu}^{n+1}$ and $R_{\phi}^{n+1}$.

Similarly, for the first term in $A_5$, we have
\begin{equation}\label{A5_1}
\begin{aligned}
&-\frac{1}{\delta}(G'(\psi(t^{n}))-G'(\psi^{n}), e_{\psi}^{n+1}-e_{\psi}^{n})_{\Gamma}\\
&=-\frac{\tau}{\delta}(\tilde{H}^n, \frac{e_{\psi}^{n+1}-e_{\psi}^{n}}{\tau})_{\Gamma}
=-\frac{\tau}{\delta}(\tilde{H}^n, \Delta_{\Gamma} e_{\Gamma}^{n+1}-\partial_{\mathbf{n}}e_{\mu}^{n+1}+R_{\psi}^{n+1})_{\Gamma}\\
&=\frac{\tau}{\delta}(\nabla_{\Gamma}\tilde{H}^n,\nabla_{\Gamma} e_{\Gamma}^{n+1})_{\Gamma}+\frac{\tau}{\delta}(\tilde{H}^n, \partial_{\mathbf{n}}e_{\mu}^{n+1})_{\Gamma}-\frac{\tau}{\delta}(\tilde{H}^n,R_{\psi}^{n+1})_{\Gamma}\\
&\leq \frac{\tau}{\delta} \|\nabla_{\Gamma}\tilde{H}^n\|_{\Gamma}\|\nabla_{\Gamma}e_{\Gamma}^{n+1}\|_{\Gamma}+
\frac{\tau}{\delta}\| \tilde{H}^n\|_{\Gamma}\|\partial_{\mathbf{n}}e_{\mu}^{n+1}\|_{\Gamma}+
\frac{\tau}{\delta}\| \tilde{H}^n\|_{\Gamma}\|R_{\psi}^{n+1}\|_{\Gamma}\\
&\leq C_{10}\tau( \|e_{\psi}^{n}\|_{\Gamma}+\|\nabla_{\Gamma} e_{\psi}^{n}\|_{\Gamma})\|\nabla_{\Gamma}e_{\Gamma}^{n+1}\|_{\Gamma}+C_{11}\tau\|e_{\psi}^{n}\|_{\Gamma}
\|\partial_{\mathbf{n}}e_{\mu}^{n+1}\|_{\Gamma}
+C_{12}\tau\|e_{\psi}^{n}\|_{\Gamma}\|R_{\psi}^{n+1}\|_{\Gamma}\\
&\leq C_{13} \tau^3 +\frac{\tau}{4}\|\nabla_{\Gamma}e_{\Gamma}^{n+1}\|_{\Gamma}^2+\frac{K\tau}{32}
\|\partial_{\mathbf{n}}e_{\mu}^{n+1}\|_{\Gamma}^2,
\end{aligned}
\end{equation}
where $C_i$ ($i=10,11,12,13$) are constants independent of $\tau$.
Here, we use the assumption that $e_{\psi}^{n}$ satisfies the error estimate \eqref{finalestimate} and use the
estimate for $R_{\psi}^{n+1}$.

For the second term in $A_5$, we have
\begin{equation}\label{A5_2}
\begin{aligned}
&-(R_{\Gamma}^{n+1},e_{\psi}^{n+1}-e_{\psi}^{n})_{\Gamma}
=-\tau(R_{\Gamma}^{n+1},\frac{e_{\psi}^{n+1}-e_{\psi}^{n}}{\tau})_{\Gamma}\\
&=-\tau(R_{\Gamma}^{n+1}, \Delta_{\Gamma} e_{\Gamma}^{n+1}-\partial_{\mathbf{n}}e_{\mu}^{n+1}+R_{\psi}^{n+1})_{\Gamma}\\
&=\tau(\nabla_{\Gamma} R_{\Gamma}^{n+1}, \nabla_{\Gamma} e_{\Gamma}^{n+1})_{\Gamma}+\tau(R_{\Gamma}^{n+1}, \partial_{\mathbf{n}}e_{\mu}^{n+1})_{\Gamma}-\tau(R_{\Gamma}^{n+1},R_{\psi}^{n+1})_{\Gamma}\\
&\leq 2\tau\|\nabla_{\Gamma} R_{\Gamma}^{n+1}\|_{\Gamma}^2+\frac{\tau}8\|\nabla_{\Gamma} e_{\Gamma}^{n+1}\|_{\Gamma}^2+\frac{\tau}2\| R_{\Gamma}^{n+1}\|_{\Gamma}^2+\frac{\tau}2\| R_{\psi}^{n+1}\|_{\Gamma}^2\\
&+\frac{8\tau}{K}\| R_{\Gamma}^{n+1}\|_{\Gamma}^2+\frac{\tau K}{32}\|\partial_{\mathbf{n}}e_{\mu}^{n+1}\|_{\Gamma}^2\\
&\leq C_{14} \tau^3+\frac{\tau}8\|\nabla_{\Gamma} e_{\Gamma}^{n+1}\|_{\Gamma}^2+\frac{\tau K}{32}\|\partial_{\mathbf{n}}e_{\mu}^{n+1}\|_{\Gamma}^2,
\end{aligned}
\end{equation}
where $C_{14}$ is a constant independent of $\tau$ and we use the estimates for $R_{\psi}^{n+1}$ and $R_{\Gamma}^{n+1}$.

Combine \eqref{errorzong} with \eqref{A1}, \eqref{A2}, \eqref{A3}, \eqref{A4_1_2}, \eqref{A4_2_1}, \eqref{A5_1} and \eqref{A5_2}, we derive
\begin{equation}\label{errorzong2}
\begin{aligned}
&\frac{\varepsilon}2(\| \nabla e_{\phi}^{n+1}\|_{\Omega}^2-\|\nabla e_{\phi}^{n}\|_{\Omega}^2+\| \nabla e_{\phi}^{n+1}-\nabla e_{\phi}^{n}\|_{\Omega}^2)
+\frac{\varepsilon}2(\| e_{\phi}^{n+1}\|_{\Omega}^2-\| e_{\phi}^{n}\|_{\Omega}^2+\| e_{\phi}^{n+1}- e_{\phi}^{n}\|_{\Omega}^2)\\
&+\frac{\delta\kappa}2(\| \nabla_{\Gamma} e_{\psi}^{n+1}\|_{\Gamma}^2-\|\nabla_{\Gamma} e_{\psi}^{n}\|_{\Gamma}^2+\| \nabla_{\Gamma} e_{\psi}^{n+1}-\nabla_{\Gamma} e_{\psi}^{n}\|_{\Gamma}^2)\\
&+\frac{\varepsilon}2(\| e_{\psi}^{n+1}\|_{\Gamma}^2-\| e_{\psi}^{n}\|_{\Gamma}^2+\| e_{\psi}^{n+1}- e_{\psi}^{n}\|_{\Gamma}^2)
+s_1\|e_{\phi}^{n+1}-e_{\phi}^{n}\|_{\Omega}^2+s_2\|e_{\psi}^{n+1}-e_{\psi}^{n}\|_{\Gamma}^2\\
&+\frac{\tau}{2}\| \nabla e_{\mu}^{n+1}\|_{\Omega}^2
+\frac{\tau}{2} \| \nabla_{\Gamma} e_{\Gamma}^{n+1}\|_{\Gamma}^2+\frac{27K \tau}{32}\|\partial_{\mathbf{n}}e_{\mu}^{n+1}\|_{\Gamma}^2\\
&\leq C_{15} \tau^3+ C_{16}\tau(\| \nabla e_{\phi}^{n+1}\|_{\Omega}^2+\|e_{\phi}^{n+1}\|_{\Omega}^2+\| e_{\phi}^{n+1}- e_{\phi}^{n}\|_{\Omega}^2\\
&+\| \nabla_{\Gamma} e_{\psi}^{n+1}\|_{\Gamma}^2+\|e_{\psi}^{n+1}\|_{\Gamma}^2+\| e_{\psi}^{n+1}- e_{\psi}^{n}\|_{\Gamma}^2).
\end{aligned}
\end{equation}
Here, $C_{15}$ is a constant independent of $\tau$ and the constant $C_{16}=\max\{\varepsilon/2+2\varepsilon^2, \delta\kappa/2+2\varepsilon^2, s_1/2, s_2/2\}$, which is also independent of $\tau$.

Summing \eqref{errorzong2} together for $n=0$ to $m$, we derive
\begin{equation}\label{sum}
\begin{aligned}
&\frac{\varepsilon}2\| \nabla e_{\phi}^{m+1}\|_{\Omega}^2+\frac{\varepsilon}2\| e_{\phi}^{m+1}\|_{\Omega}^2+\frac{\delta\kappa}2\| \nabla_{\Gamma} e_{\psi}^{m+1}\|_{\Gamma}^2+\frac{\varepsilon}2\| e_{\psi}^{m+1}\|_{\Gamma}^2\\
&+\sum_{n=0}^m \bigg{(} \frac{\varepsilon}2\| \nabla e_{\phi}^{n+1}-\nabla e_{\phi}^{n}\|_{\Omega}^2 +(s_1+\frac{\varepsilon}2)\| e_{\phi}^{n+1}- e_{\phi}^{n}\|_{\Omega}^2\\
&+\frac{\delta\kappa}2 \| \nabla_{\Gamma} e_{\psi}^{n+1}-\nabla_{\Gamma} e_{\psi}^{n}\|_{\Gamma}^2 +(s_2+\frac{\varepsilon}2)\| e_{\psi}^{n+1}- e_{\psi}^{n}\|_{\Gamma}^2\\
&+\frac{\tau}{2}\|\nabla e_{\mu}^{n+1}\|_{\Omega}^2+\frac{\tau}2\|\nabla_{\Gamma} e_{\Gamma}^{n+1}\|_{\Gamma}^2  +\frac{27K\tau}{32}\|\partial_{\mathbf{n}}e_{\mu}^{n+1}\|_{\Gamma}^2 \bigg{)}\\
&\leq C_{15} (m+1)\tau^3+C_{16} \tau\sum_{n=0}^m\bigg{(}\| \nabla e_{\phi}^{n+1}\|_{\Omega}^2+\|e_{\phi}^{n+1}\|_{\Omega}^2+\| e_{\phi}^{n+1}- e_{\phi}^{n}\|_{\Omega}^2\\
&+\|\nabla_{\Gamma} e_{\psi}^{n+1}\|_{\Gamma}^2+\|e_{\psi}^{n+1}\|_{\Gamma}^2+\| e_{\psi}^{n+1}- e_{\psi}^{n}\|_{\Gamma}^2\bigg{)},
\end{aligned}
\end{equation}

Denote
\begin{equation}\label{zetadef}
\omega = \min\{ \frac{\varepsilon}2, \frac{\delta\kappa}2, (s_1+\frac{\varepsilon}2), (s_2+\frac{\varepsilon}2)   \},
\end{equation}

\begin{equation}\label{Idef}
\begin{aligned}
I_m&=\frac{\varepsilon}2\| \nabla e_{\phi}^{m+1}\|_{\Omega}^2+\frac{\varepsilon}2\| e_{\phi}^{m+1}\|_{\Omega}^2+\frac{\delta\kappa}2\| \nabla_{\Gamma} e_{\psi}^{m+1}\|_{\Gamma}^2+\frac{\varepsilon}2\| e_{\psi}^{m+1}\|_{\Gamma}^2\\
&+(s_1+\frac{\varepsilon}2)\| e_{\phi}^{m+1}- e_{\phi}^{m}\|_{\Omega}^2+(s_2+\frac{\varepsilon}2)\| e_{\psi}^{m+1}- e_{\psi}^{m}\|_{\Gamma}^2,
\end{aligned}
\end{equation}
and
\begin{equation}\label{Sdef}
\begin{aligned}
S_m&=\sum_{n=0}^m \bigg{(} \frac{\varepsilon}2\| \nabla e_{\phi}^{n+1}-\nabla e_{\phi}^{n}\|_{\Omega}^2+\frac{\delta\kappa}2 \| \nabla_{\Gamma} e_{\psi}^{n+1}-\nabla_{\Gamma} e_{\psi}^{n}\|_{\Gamma}^2\\
&+ \frac{\tau}{2}\|\nabla e_{\mu}^{n+1}\|_{\Omega}^2+\frac{\tau}2\|\nabla_{\Gamma} e_{\Gamma}^{n+1}\|_{\Gamma}^2  +\frac{27K\tau}{32}\|\partial_{\mathbf{n}}e_{\mu}^{n+1}\|_{\Gamma}^2\bigg{)}.
\end{aligned}
\end{equation}

Then we have
\begin{equation}\label{gronwall}
\begin{aligned}
I_m+S_m&\leq C_{15} T \tau^2+C_{16}\tau\sum_{n=0}^m \bigg{(}\| \nabla e_{\phi}^{n+1}\|_{\Omega}^2+\|e_{\phi}^{n+1}\|_{\Omega}^2+\| e_{\phi}^{n+1}- e_{\phi}^{n}\|_{\Omega}^2\\
&+\|\nabla_{\Gamma} e_{\psi}^{n+1}\|_{\Gamma}^2+\|e_{\psi}^{n+1}\|_{\Gamma}^2+\| e_{\psi}^{n+1}- e_{\psi}^{n}\|_{\Gamma}^2\bigg{)}\\
& = C_{15} T \tau^2+\frac{C_{16}}{\omega} \tau \sum_{n=0}^m \omega\bigg{(}\| \nabla e_{\phi}^{n+1}\|_{\Omega}^2+\|e_{\phi}^{n+1}\|_{\Omega}^2+\| e_{\phi}^{n+1}- e_{\phi}^{n}\|_{\Omega}^2\\
&+\|\nabla_{\Gamma} e_{\psi}^{n+1}\|_{\Gamma}^2+\|e_{\psi}^{n+1}\|_{\Gamma}^2+\| e_{\psi}^{n+1}- e_{\psi}^{n}\|_{\Gamma}^2\bigg{)}\\
&\leq C_{15} T \tau^2+C_{17}\tau\sum_{n=0}^m I_n.
\end{aligned}
\end{equation}
where $C_{17}=C_{16}/\omega$ is a constant independent of $\tau$.
According to the discrete Gronwall's inequality, there exists some constants $\tilde{c}_0$, which is independent of $\tau$, and $\tau_0 = 1/C_{17}$, such that, when $\tau < \tau_0$,
\begin{equation}\label{gronwall2}
I_m+S_m\leq\tilde{c}_0\tau^2.
\end{equation}
And thus the error estimate \eqref{finalestimate} holds for $e_{\phi}^{m+1}$ and $e_{\psi}^{m+1}$.

\end{proof}

\begin{remark}
If we set the parameters as
$$
\varepsilon = \delta = 0.02, \quad \kappa = 1, \quad s_1 = s_2 = 50,
$$
then $C_{16} = s_1/2=25$, $\omega = \varepsilon/2 = 0.01$, we obtain $C_{17} = 2500$ and $\tau_0 = 4 \times 10^{-4}$. Namely, when $\tau < 4 \times 10^{-4}$, the numerical solutions satisfy the error estimates \eqref{finalestimate}.
Hence, the error estimates are applicable for time increments that can be used in practical simulations.
\end{remark}

\section{Numerical simulations\label{s5}}

In this section, we present numerical experiments of the KLLM model (Eq. \eqref{CHK}) by implementing
the developed scheme \eqref{SIscheme1}-\eqref{SIscheme6}. The numerical examples include the simulations with different energy potentials,
the comparison with the numerical results in \cite{knopf2020}, accuracy tests with respect to the time step size, and the convergence of discrete solutions for $K\rightarrow \infty$ and $K\rightarrow 0$.
%

In this section, we present the numerical simulations in two dimensions.
For the spatial operators, we use the second-order central finite difference
method to discretize them over a uniform spatial grid. After the spacial discretization, the generalized minimum residual method is used as the linear solver in this section.

\begin{remark}
In this section, we conduct experiments on the rectangular domain. For more regular domains, the strategy is similar. The finite difference method for the bulk descretization is the same. For the boundary conditions, we need to choose a suitable coordinate for the boundary $\Gamma$ so that we can get the specific representation of the operator $\Delta_{\Gamma}$ and $\mathbf{n}$. And then, we can use the finite difference method for the spatial discretization on the boundary.
Numerical experiments on other two-dimensional domains will be our future work.
\end{remark}

\subsection{Case with Flory-Huggins potential}
In this section, we consider the numerical approximations for the KLLM model with the logarithmic Flory-Huggins potential. Namely, for the bulk and surface potential, we consider the logarithmic Flory-Huggins potential as follows,
\begin{equation}\label{FL1}
F(\phi)=\phi \ln \phi + (1-\phi)\ln(1-\phi) + \theta\phi(1-\phi),
\end{equation}
\begin{equation}\label{FL2}
G(\psi)=\psi \ln \psi + (1-\psi)\ln(1-\psi) + \theta\psi(1-\psi),
\end{equation}
where the constant $\theta>1$.

\begin{remark}
The Cahn-Hilliard type equation with the Flory-Huggins potential is widely used to describe the spinodal decomposition and coarsening phenomena of binary mixtures. In this case, instead of treating $\phi$ and $\psi$ as the order parameters, $\phi$ and $\psi$ denote the mass concentration of one component in the bulk and on the boundary respectively. And the mass concentration of the other component in the bulk and on the boundary are denoted by $1-\phi$ and $1-\psi$ respectively. Hence, the corresponding physical relevant interval is $(0, 1)$.

\end{remark}

Following the work in \cite{flory}, we consider the regularized logarithmic potential in this section. Precisely, for $0 < \zeta \ll 1$, the regularized potential is
\begin{equation}\label{regularFL1}
\hat{F}(\phi)=\left\{\begin{aligned}
&\phi\ln\phi+\frac{(1-\phi)^2}{2\zeta}+(1-\phi)\ln \zeta-\frac{\zeta}2+\theta\phi(1-\phi) \qquad \phi>1-\zeta,\\
&\phi\ln\phi+(1-\phi)\ln (1-\phi)+\theta\phi(1-\phi) \qquad \zeta\leq\phi\leq1-\zeta,\\
&(1-\phi)\ln (1-\phi)+\frac{\phi^2}{2\zeta}+\phi\ln \zeta-\frac{\zeta}2+\theta\phi(1-\phi) \qquad \phi<\zeta.
\end{aligned}
\right.
\end{equation}
\begin{equation}\label{regularFL2}
\hat{G}(\psi)=\left\{\begin{aligned}
&\psi\ln\psi+\frac{(1-\psi)^2}{2\zeta}+(1-\psi)\ln \zeta-\frac{\zeta}2+\theta\psi(1-\psi) \qquad \psi>1-\zeta,\\
&\psi\ln\psi+(1-\psi)\ln (1-\psi)+\theta\psi(1-\psi) \qquad \zeta\leq\psi\leq1-\zeta,\\
&(1-\psi)\ln (1-\psi)+\frac{\psi^2}{2\zeta}+\psi\ln \zeta-\frac{\zeta}2+\theta\psi(1-\psi) \qquad \psi<\zeta.
\end{aligned}
\right.
\end{equation}
The advantages of using the regularized potential is that the domains for the regularized potential $\hat{F}$ and $\hat{G}$ are $\mathbb{R}$, and thus, there's no need to worry about the overflow which could be caused by any small fluctuation near the domain boundary $(0, 1)$ of the numerical solution. Obviously, the second derivatives of $\hat{F}$ and $\hat{G}$ with respect to $\phi$ and $\psi$ are bounded, respectively.

\begin{figure}
\centering
\includegraphics[scale=0.25]{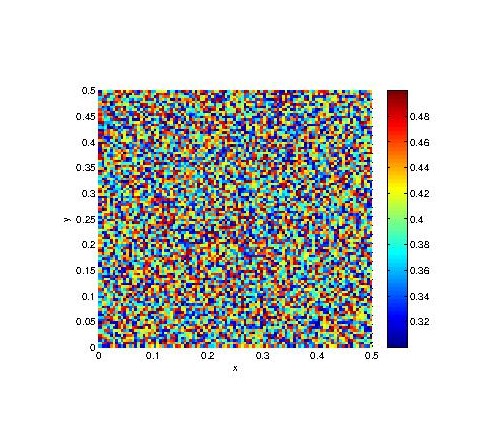}
\label{diffkinitial}
\caption{The initial data of the Section 5.1 and 5.2.}
\end{figure}

We conduct numerical simulations from $t=0$ to $T=0.1$ on the domain $\Omega=(0,0.5)^2\subset\mathbb{R}^2$ with the spatial step size $h=0.005$ and the time step $\tau=1e-4$. The initial data is set as random values between 0.4 and 0.6, as shown in Fig. 1. The parameters are set as
$$
\varepsilon=\delta=0.05, \  \kappa=1, \  s_1=s_2=500, \ \zeta=0.005, \ \theta=2.5.
$$
The numerical results of the KLLM model at $t=0.005$, $t=0.01$, $t=0.02$ and $t=0.05$ for
different $K$ ($K = 0.1,1,10$) are plotted in Fig. 2. Note that the phases are separated in all the cases and it shows different phenomenon on the boundary for different $K$. The time evolution of the energy and mass is plotted in Fig. 3 and 4 respectively. It reveals that the numerical scheme is energy stable. And the bulk and surface mass change with respect to time but the sum of them, namely, the total mass, is conserved for different $K$, which is consistent with the analysis in Section 3.

\begin{figure}
\centering
\includegraphics[scale=0.19]{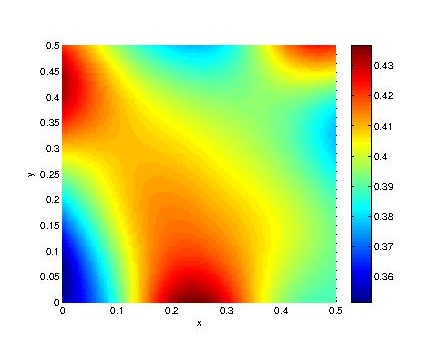}
\includegraphics[scale=0.19]{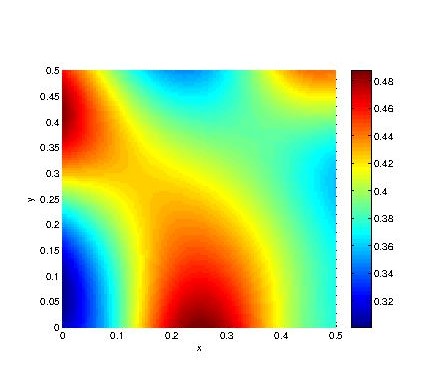}
\includegraphics[scale=0.19]{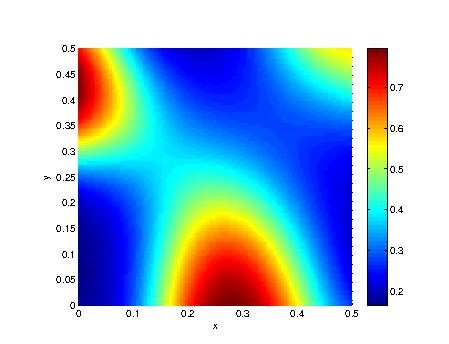}
\includegraphics[scale=0.19]{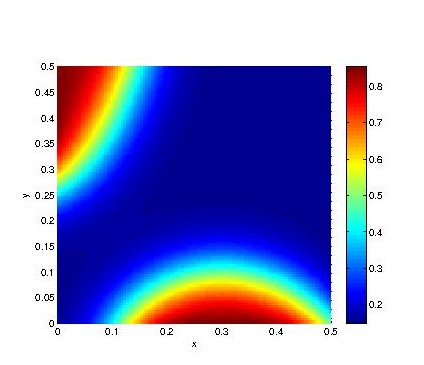}\\
\includegraphics[scale=0.19]{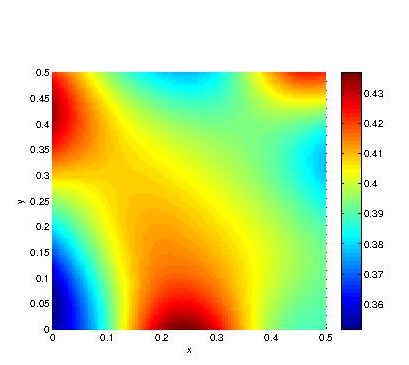}
\includegraphics[scale=0.19]{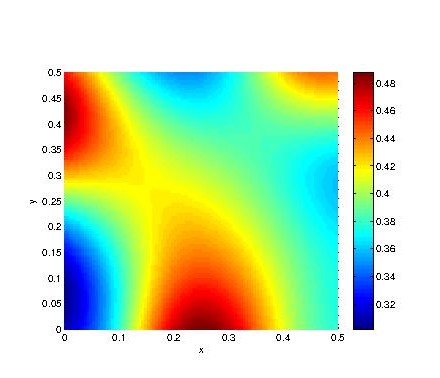}
\includegraphics[scale=0.19]{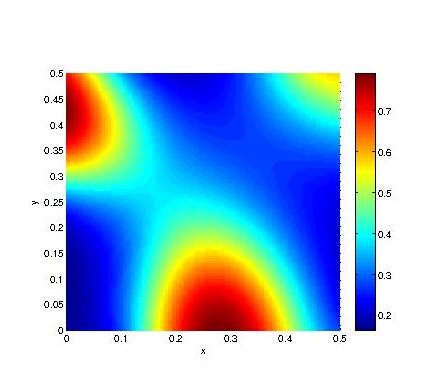}
\includegraphics[scale=0.19]{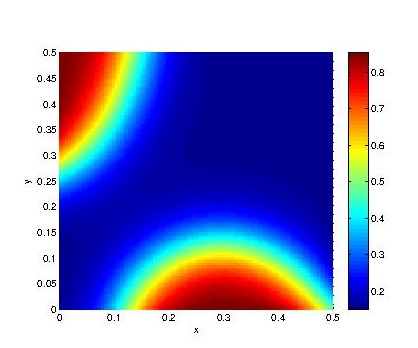}\\
\includegraphics[scale=0.19]{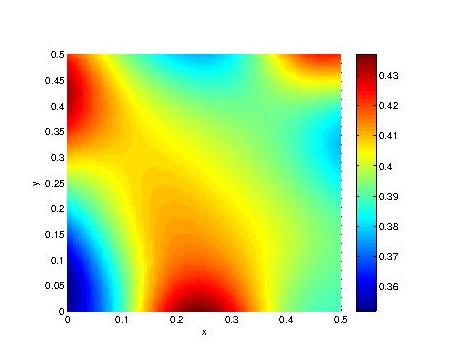}
\includegraphics[scale=0.19]{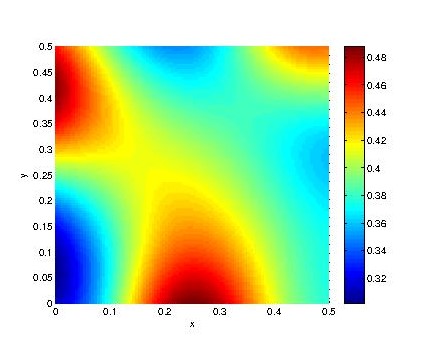}
\includegraphics[scale=0.19]{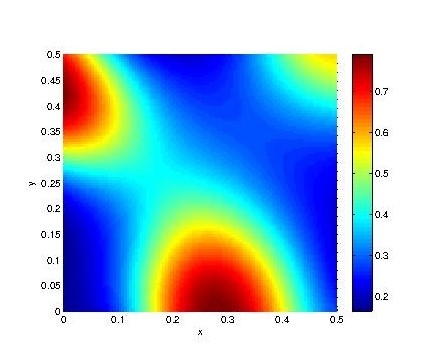}
\includegraphics[scale=0.19]{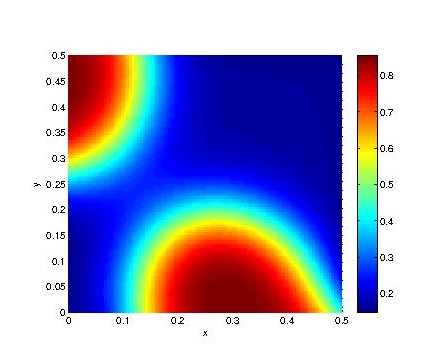}\\
\label{diffkevolutionflory}
\caption{Numerical results of the KLLM model with the Flory-Huggins potential at $t=0.005$, $t=0.01$, $t=0.02$ and $t=0.05$. From top to bottom: $K=0.1$, $K=1$ and $K=10$.}
\end{figure}
\begin{figure}
\centering
\includegraphics[scale=0.35]{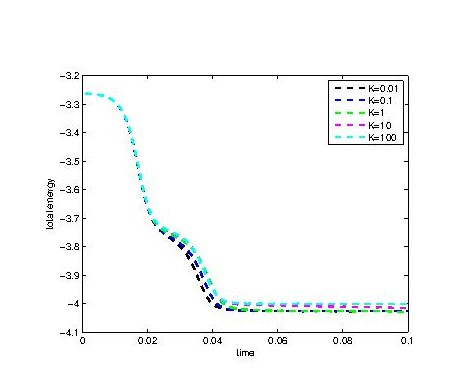}
\label{energyplotflory}
\caption{Energy evolution of the KLLM model with the initial data shown in Fig. 1 with the Flory-Huggins potential.}
\end{figure}
\begin{figure}
\centering
\includegraphics[scale=0.25]{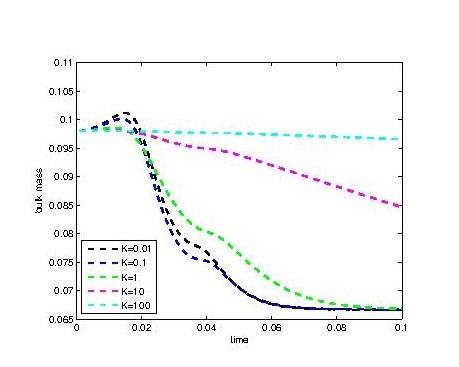}
\includegraphics[scale=0.25]{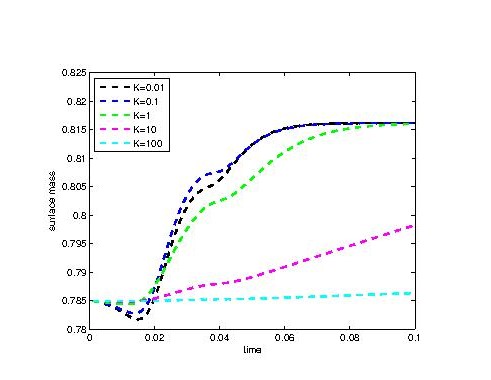}
\includegraphics[scale=0.25]{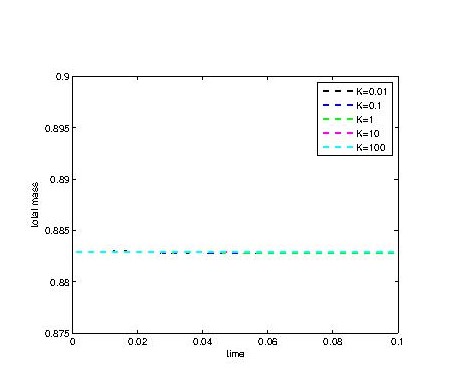}
\label{massplotflory}
\caption{Mass evolution of the KLLM model with the initial data shown in Fig. 1 and the Flory-Huggins potential: the bulk mass evolution (left), the surface mass evolution (middle) and the total mass (right).}
\end{figure}

The minimal and maximal occurring values of $\phi$ and $\psi$ for different $K$ are plotted in Figs. 5-6. We can conclude that in this case, the values of $\phi$ and $\psi$ lie in the physical relevant interval $(0,1)$, indicating the practicality of the proposed scheme.
\begin{figure}
\centering
\includegraphics[scale=0.35]{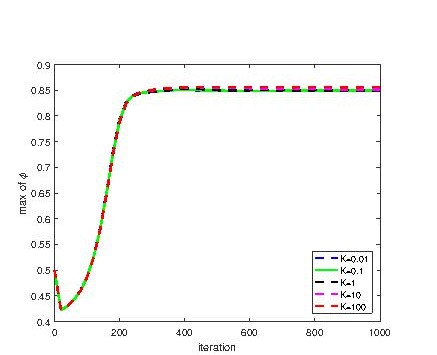}
\includegraphics[scale=0.35]{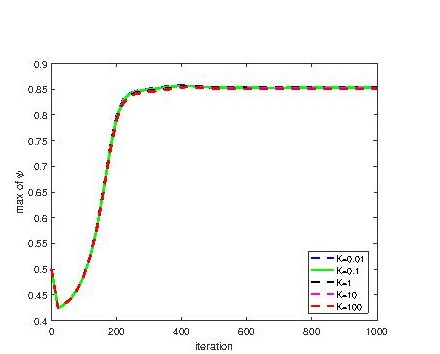}
\label{maxflory}
\caption{ Maximum values of $\phi$ (left) and $\psi$ (right) with respect to iterations with the initial data shown in Fig. 1 and the Flory-Huggins potential.}
\end{figure}


\begin{figure}
\centering
\includegraphics[scale=0.35]{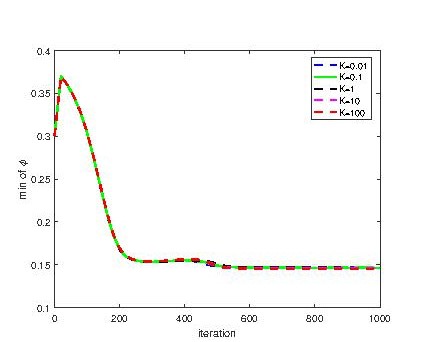}
\includegraphics[scale=0.35]{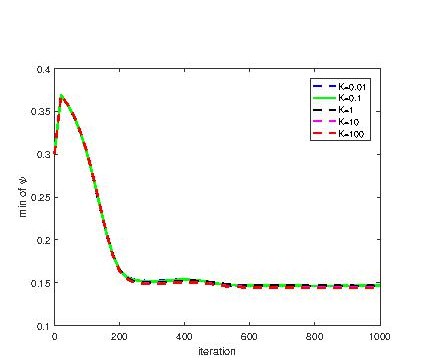}
\label{minflory}
\caption{ Minimum values of $\phi$ (left) and $\psi$ (right) with respect to iterations  with the initial data shown in Fig. 1 and the Flory-Huggins potential.}
\end{figure}


\subsection{Case with the double-well potential}
In this section, we consider the case with the modified double-well potential shown in Eq. \eqref{modifyFG}. Precisely, we choose $F(\phi)$ and $G(\psi)$ as follows:
\begin{equation}
F(\phi)=\left\{\begin{aligned}
&(\phi-1)^2 \qquad \phi>1,\\
&\frac{1}4(\phi^2-1)^2 \quad -1\leq\phi\leq1,\\
&(\phi+1)^2 \qquad \phi<-1,
\end{aligned}
\right.
\quad
G(\psi)=\left\{\begin{aligned}
&(\psi-1)^2 \qquad \psi>1,\\
&\frac{1}4(\psi^2-1)^2 \quad -1\leq\psi\leq1,\\
&(\psi+1)^2 \qquad \psi<-1.
\end{aligned}
\right.
\end{equation}
Obviously, Remark 5 shows that the second derivative of $F$ with respect to $\phi$ and the second derivative of $G$ with respect to $\psi$, namely, $F''$ and $G''$, are Lipschitz and bounded.

\begin{remark}
For the case with the modified double-well potential, in order to describe the binary alloys, $\phi$ and $\psi$ are treated as the order parameters, denoting the difference of two local relative concentrations. The regions with $\phi = \pm 1$ (or $\psi = \pm 1$) in the domain $\Omega$ (or on the boundary $\Gamma$) represent the pure phases of the materials. Hence, the corresponding physical relevant interval is $[-1, 1]$.
\end{remark}

We conduct numerical simulations from $t=0$ to $T=1$ on the domain $\Omega=(0,0.5)^2\subset\mathbb{R}^2$ with the spatial step size $h=0.005$ and the time step $\tau=1e-4$. The initial data is set as random values between 0.4 and 0.6, as shown in Fig. 1. And the parameters are set as
$$
\varepsilon=\delta=0.02, \  \kappa=1, \  s_1=s_2=50.
$$

The numerical solutions of the KLLM model at time $t=10^{-3}$, $5\times 10^{-3}$, $10^{-2}$ and $5\times 10^{-2}$ for different $K$ ($K=0.01, 1, 100$) are plotted in Fig. 7. It shows the separation of phases and there exists interesting phenomenon on the boundary for different $K$.
The energy evolution from $t=0$ to $T=1$ and its local magnification from $t=0$ to $t=4\times 10^{-3}$ are plotted in Fig. 8, revealing the energy stability.
From the magnification, it reveals that at the beginning, the energy decreases faster for smaller $K$. Namely, the energy minimization benefits from low values of $K$. We can obtain the same observation from Fig. 3 and Fig. 16.
This phenomena may because high values of $K$ inhibit but low values of $K$ promote the mass transfer between $\Omega$ and $\Gamma$.
Moreover, for different $K$, the energy decrease in shape of steps, following different paths, and finally approaches approximately the same value.
The configurations of the droplet near the "steps"(at the time $t=0.1,\ 0.2,\ 0.4$) and at  the final equilibrium state (at the time $t=1$) are shown in Fig. 9, from which we may conclude that the different paths, which the decrease of the energy follows for different $K$, are related to the numbers and configurations of the droplets with values around -1. And the configurations of the droplets at the equilibrium state for different $K$ are similar.

The time evolutions of masses (the bulk mass $\int_{\Omega}\phi dx$ and the surface mass $\int_{\Gamma}\psi dS$) are plotted in Fig. 10. Note that the bulk and surface mass are not conserved respectively, but time evolution of the total mass (sum of the bulk and the surface mass) shows the conservation for different $K$, indicating the consistence between the numerical results and the analysis in Section 3.

\begin{figure}
\centering
\includegraphics[scale=0.19]{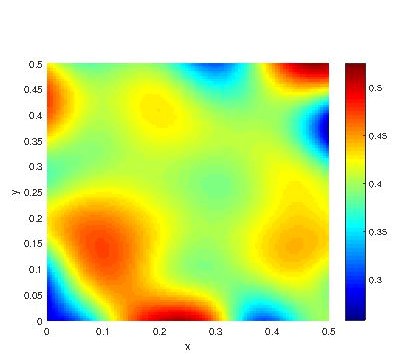}
\includegraphics[scale=0.19]{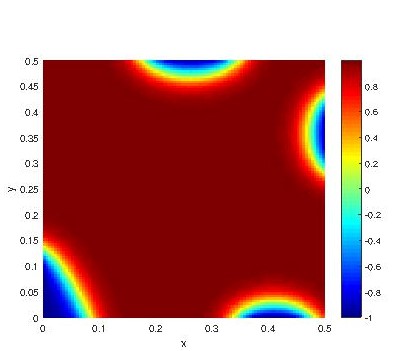}
\includegraphics[scale=0.19]{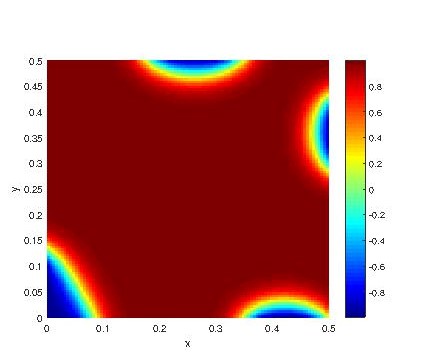}
\includegraphics[scale=0.19]{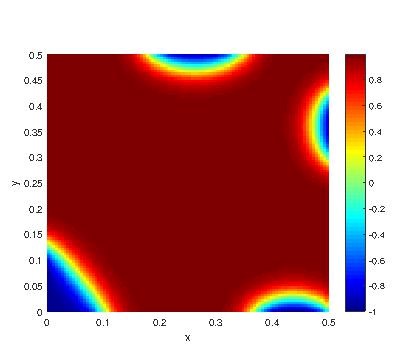}\\
\includegraphics[scale=0.19]{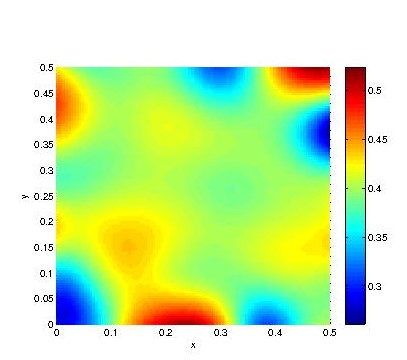}
\includegraphics[scale=0.19]{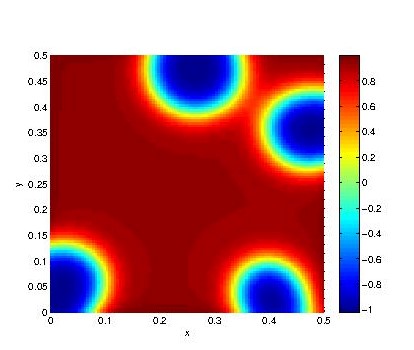}
\includegraphics[scale=0.19]{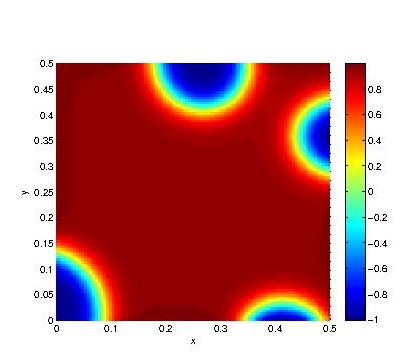}
\includegraphics[scale=0.19]{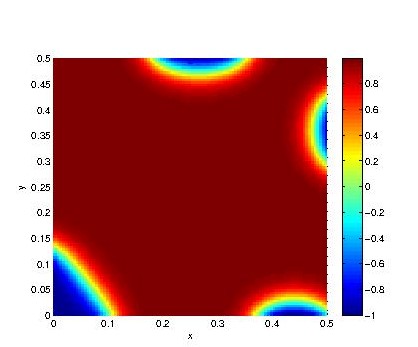}\\
\includegraphics[scale=0.19]{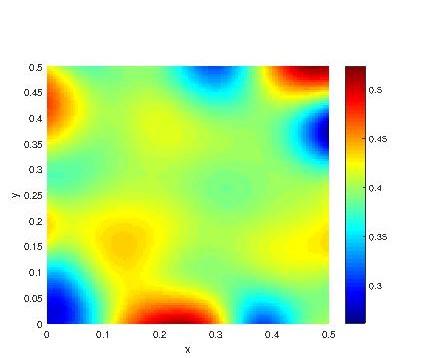}
\includegraphics[scale=0.19]{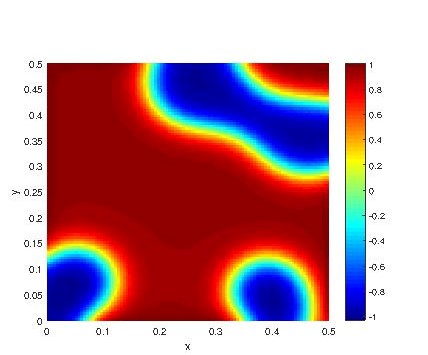}
\includegraphics[scale=0.19]{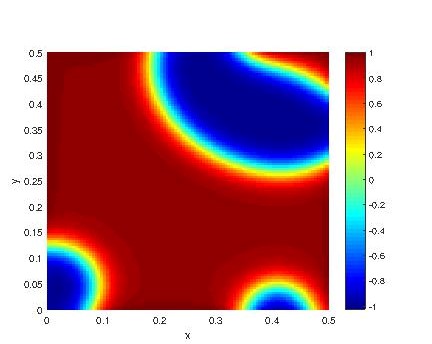}
\includegraphics[scale=0.19]{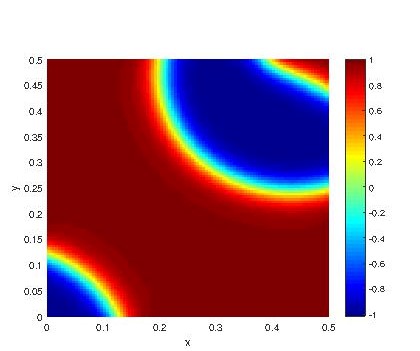}\\
\label{diffkdoublewell}
\caption{ Numerical results of the KLLM model with the initial data shown in
Fig. 1 and the modified double-well potential at time $t=10^{-3}$, $5\times 10^{-3}$, $10^{-2}$ and $5\times 10^{-2}$. From top to bottom: $K=0.01$, $K=1$ and $K=100$.  }
\end{figure}

\begin{figure}
\centering
\includegraphics[scale=0.35]{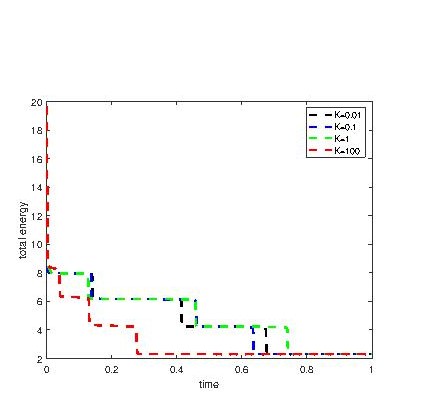}
\includegraphics[scale=0.35]{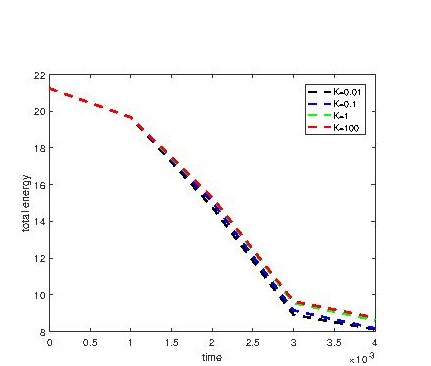}
\label{energyplot1}
\caption{Energy evolution of the KLLM model with the initial data shown in Fig. 1 and the modified double-well potential (left) and the corresponding local magnification from $t=0$ to $t=4\times 10^{-3}$ (right).}
\end{figure}

\begin{figure}
\centering
\includegraphics[scale=0.19]{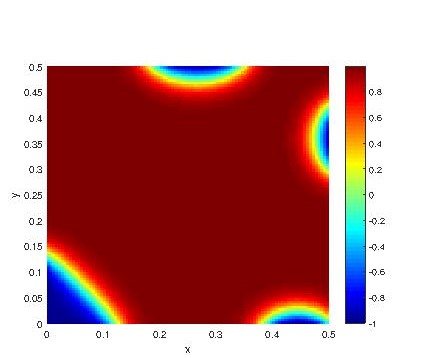}
\includegraphics[scale=0.19]{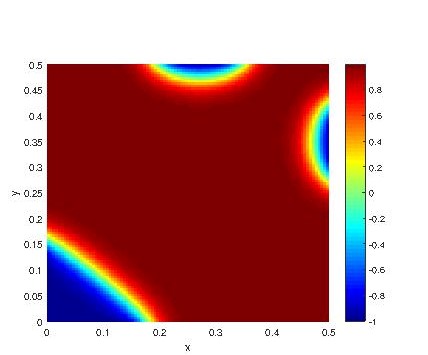}
\includegraphics[scale=0.19]{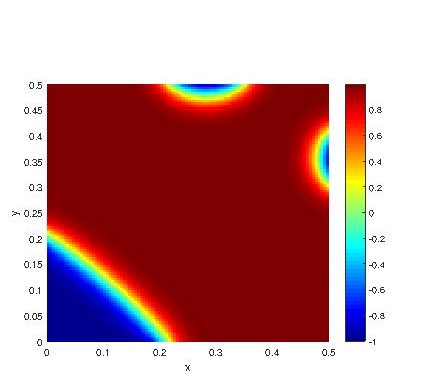}
\includegraphics[scale=0.19]{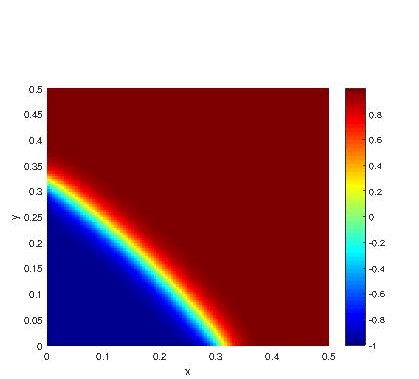}\\
\includegraphics[scale=0.19]{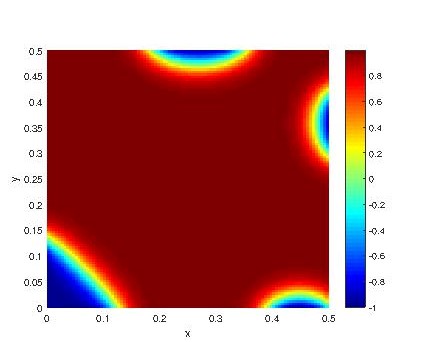}
\includegraphics[scale=0.19]{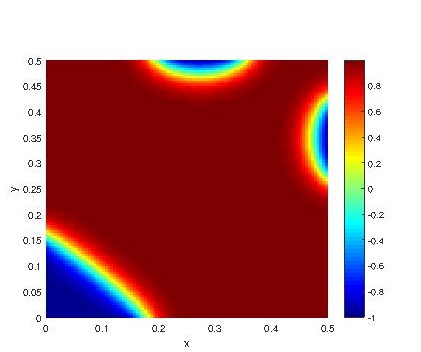}
\includegraphics[scale=0.19]{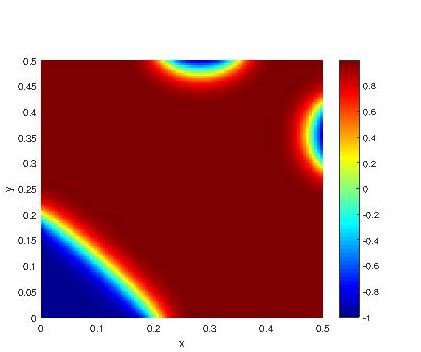}
\includegraphics[scale=0.19]{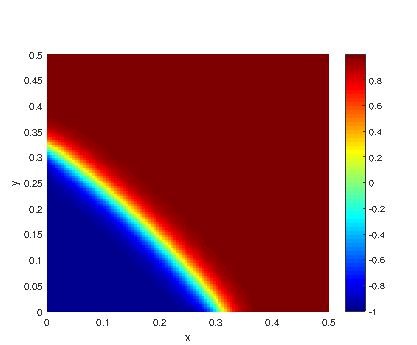}\\
\includegraphics[scale=0.19]{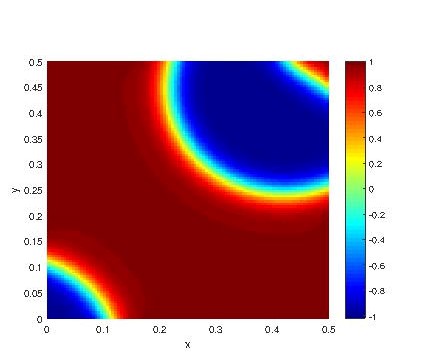}
\includegraphics[scale=0.19]{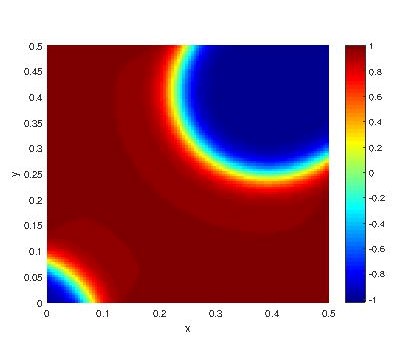}
\includegraphics[scale=0.19]{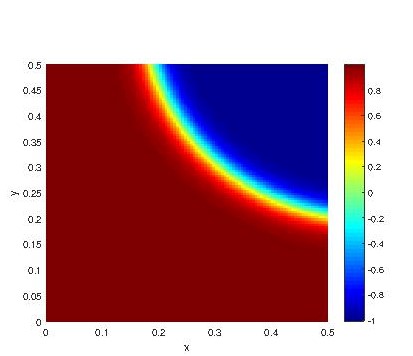}
\includegraphics[scale=0.19]{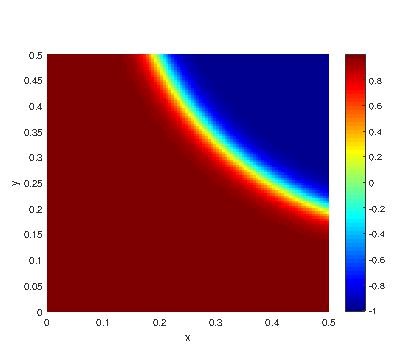}\\
\label{diffkdoublewell2}
\caption{ Numerical results of the KLLM model with the modified double-well potential at the time $t=0.1,\ 0.2,\ 0.4$ and $1$. From top to bottom: $K=0.01$, $K=1$ and $K=100$.  }
\end{figure}

\begin{figure}
\centering
\includegraphics[scale=0.25]{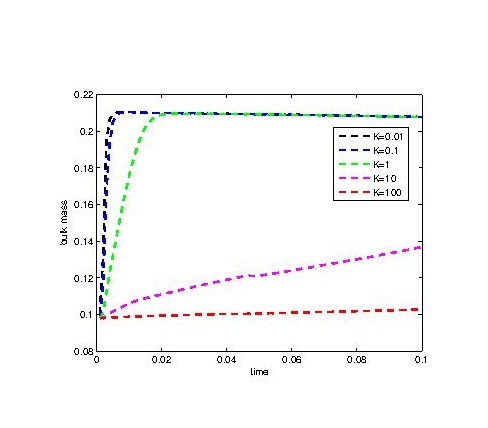}
\includegraphics[scale=0.25]{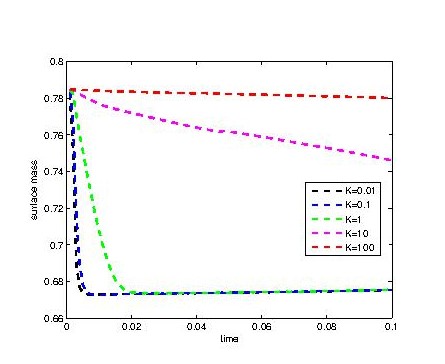}
\includegraphics[scale=0.25]{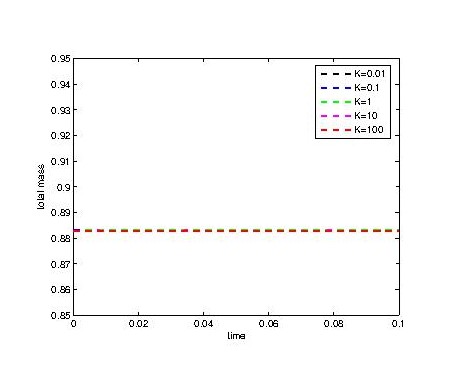}
\label{massplot1}
\caption{ Mass evolution of the KLLM model with the initial data shown in Fig. 1 and the modified double-well potential: the bulk mass evolution (left), the surface mass evolution (middle) and the total mass (right).}
\end{figure}



\begin{remark}
To the authors' knowledge, there is lack of maximum principle for the KLLM model.
Thus, theoretically, the values of $\phi$ and
$\psi$ can not be bounded in the physical relevant interval.

For the case of Flory-Huggins potential, the numerical experiments in Section 5.1 reveal that occurring values of $\phi$ and $\psi$ lie in the physical relevant interval.
For the case of the modified double-well potential, the maximal and minimal occurring values of $\phi$ and $\psi$ for different $K$ are plotted in Fig. 11 and 12.
We could conclude that the scheme proposed in this article can bound the numerical solutions
within the physical relevant interval only with some small fluctuation.
%
%

\end{remark}
\begin{figure}
\centering
\includegraphics[scale=0.35]{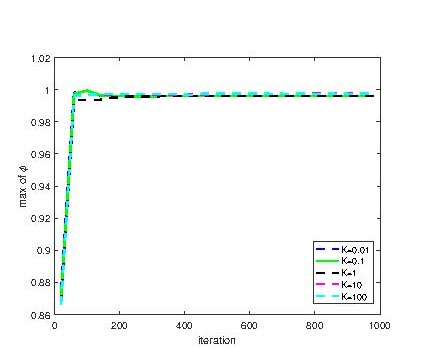}
\includegraphics[scale=0.35]{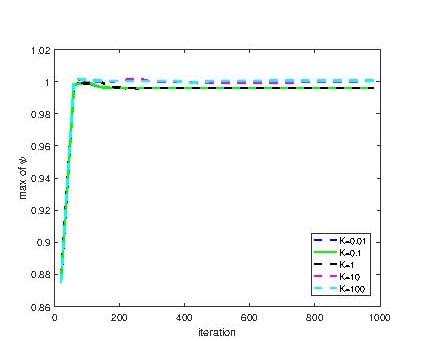}
\label{maxminplot01}
\caption{ Maximum values of $\phi$ (left) and $\psi$ (right) with respect to iterations with the initial data shown in Fig. 1 and the modified double-well potential.}
\end{figure}

\begin{figure}
\centering
\includegraphics[scale=0.35]{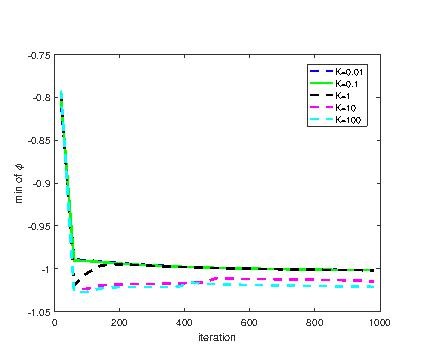}
\includegraphics[scale=0.35]{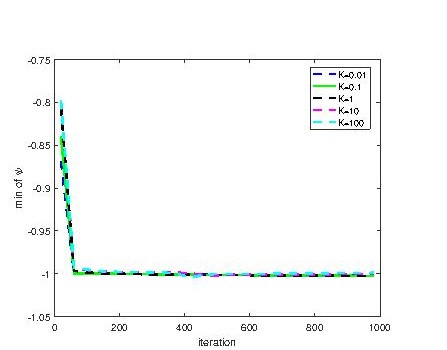}
\label{maxminplot1}
\caption{Minimum values of $\phi$ (left) and $\psi$ (right) with respect to iterations with the initial data shown in Fig. 1 and the modified double-well potential.}
\end{figure}

\subsection{Shape deformation of a droplet}

In this section, we consider the domain $\Omega=(0,1)^2\subset\mathbb{R}^2$ and place a square shaped droplet with center at $(0.5, 0.25)$ and the length of each side is 0.5 (see Fig. \ref{initial} ).
 The phase inside the droplet is set to be 1 and outside the droplet to be -1.
$F$ and $G$ are chosen to be of the regular double-well form shown in \eqref{modifyFG}.
%
And the parameters are set as
$$
\varepsilon=\delta=0.02, \  \kappa=1, \  s_1=s_2=50.
$$
We simulate the behaviour of the droplet from $t = 0$ to $T =0.1$ with the time step $\tau=2\times10^{-5}$ and the spatial step size $h=0.01$.

The evolution of the droplet is plotted in Fig. 14 for different $K$ ($K=0, 0.1, 1, 10, \infty$).
The corresponding evolution of mass and energy is plotted in Fig. \ref{mass_diffk} and Fig. \ref{energy_diffk}.
For the limiting case of $K=0$ and $K=\infty$, we use the scheme \eqref{liuwuscheme1}-\eqref{liuwuscheme6} and the scheme \eqref{GMSscheme1}-\eqref{GMSscheme5}, respectively.
In the case of the Liu-Wu model, namely, the case of $K=\infty$, the bulk mass $\int_{\Omega}\phi dx$ and the surface mass $\int_{\Gamma}\psi dS$ are conserved respectively (see Fig. \ref{mass_diffk}). Hence, in that case, the contact area on the boundary can not change. However, the square shaped droplet still evolves to attain the circular shape with constant mean curvature (see the last row in Fig. 14). When $K<\infty$, the conservation law of both the bulk and the boundary mass is relaxed and only the total mass $\int_{\Omega}\phi dx+\int_{\Gamma}\psi dS$ is conserved. Therefore, the contact area is allowed to grow (see the first four rows in Fig. 14) and the droplet's bulk mass is reduced. This phenomenon is intensifies when $K$ is decreasing. Meanwhile, the square shaped droplet also evolves to attain the circular shape when $K<\infty$. In addition, although we don't explicitly show the evolution of the total mass, we emphasize here that in our numerical experiments, the total mass is conserved for different $K$ ($K=0, 0.1, 1, 10, \infty$).

\begin{figure}
\centering
\includegraphics[scale=0.2]{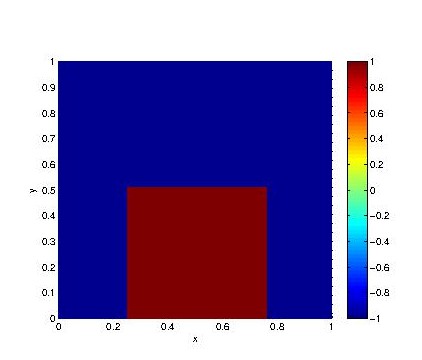}
\caption{The initial data of the square shaped droplet.}
\label{initial}
\end{figure}

\begin{figure}
\centering
\includegraphics[scale=0.19]{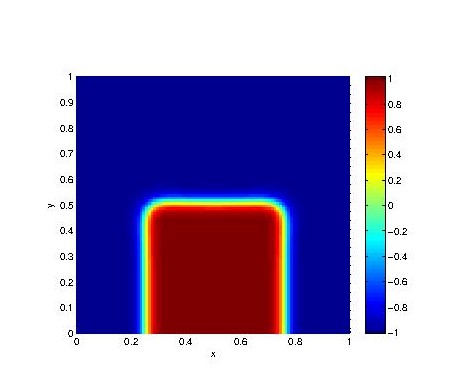}
\includegraphics[scale=0.19]{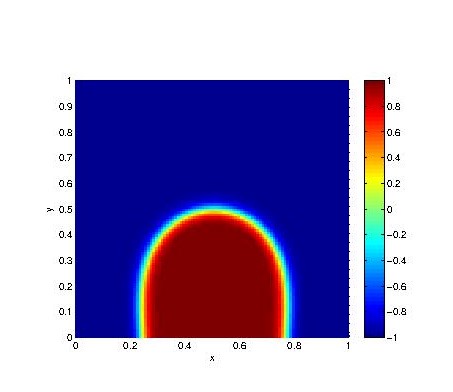}
\includegraphics[scale=0.19]{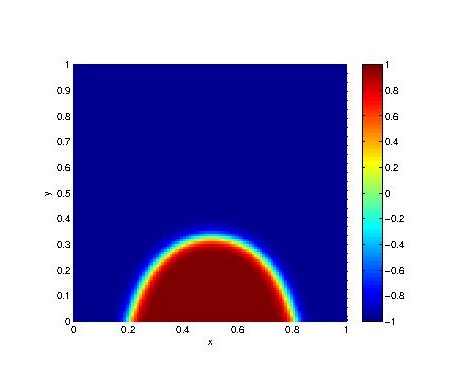}
\includegraphics[scale=0.19]{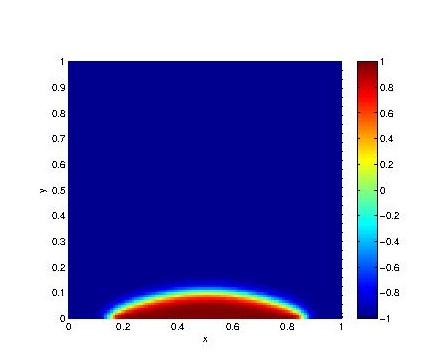}\\
\includegraphics[scale=0.19]{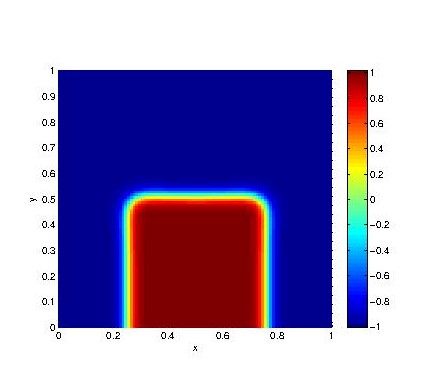}
\includegraphics[scale=0.19]{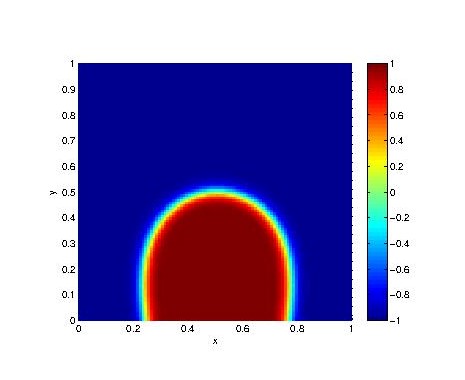}
\includegraphics[scale=0.19]{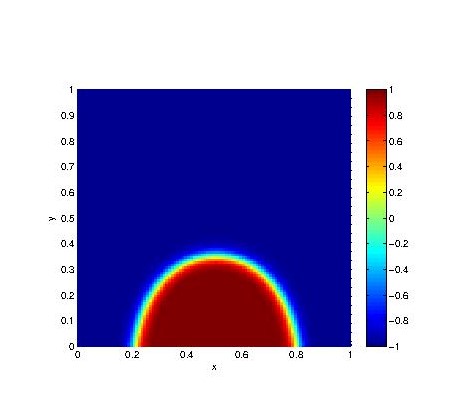}
\includegraphics[scale=0.19]{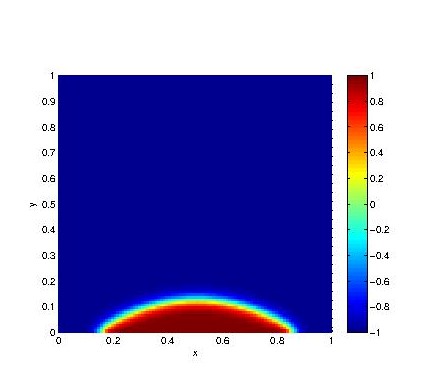}\\
\includegraphics[scale=0.19]{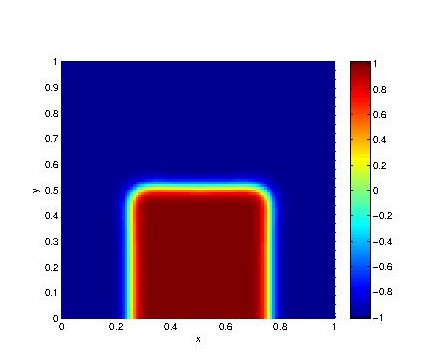}
\includegraphics[scale=0.19]{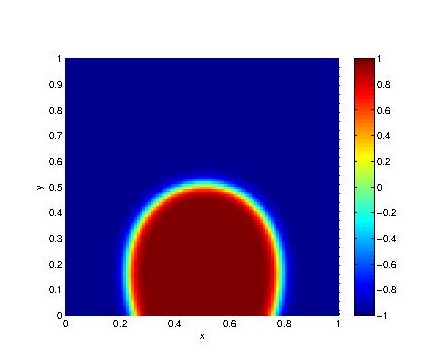}
\includegraphics[scale=0.19]{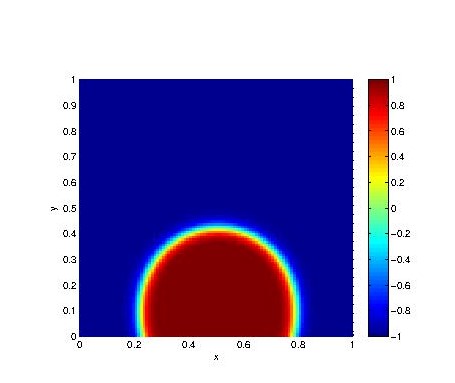}
\includegraphics[scale=0.19]{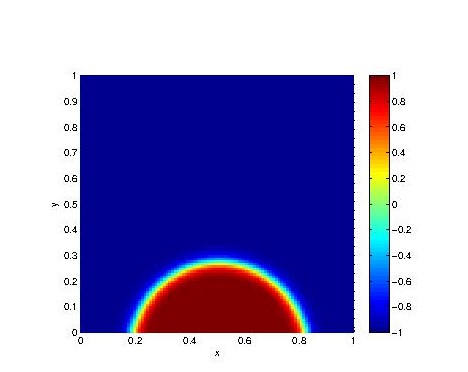}\\
\includegraphics[scale=0.19]{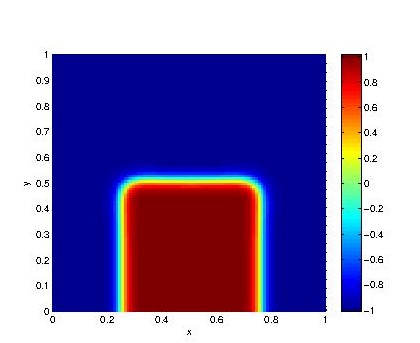}
\includegraphics[scale=0.19]{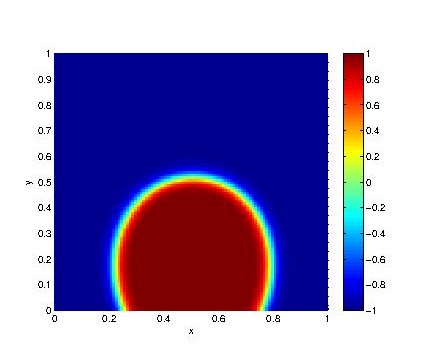}
\includegraphics[scale=0.19]{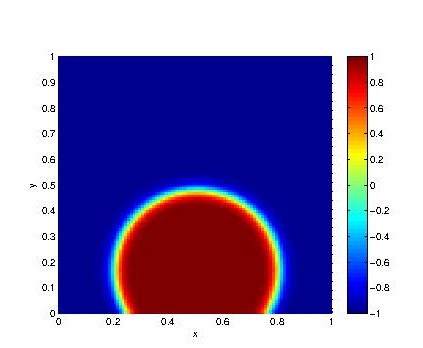}
\includegraphics[scale=0.19]{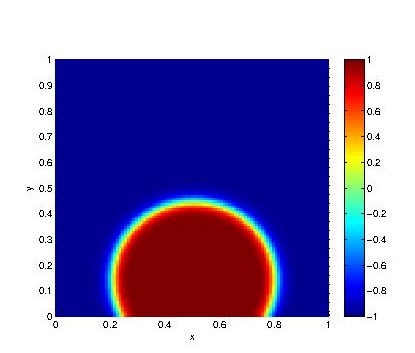}\\
\includegraphics[scale=0.19]{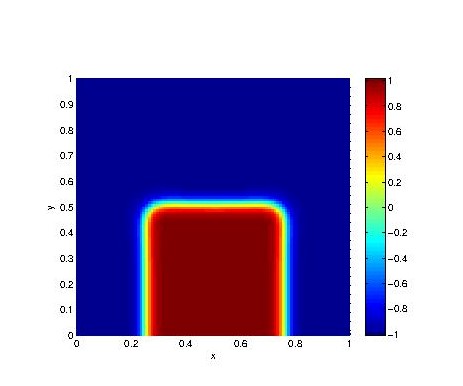}
\includegraphics[scale=0.19]{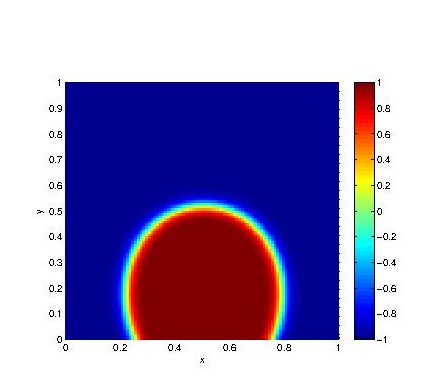}
\includegraphics[scale=0.19]{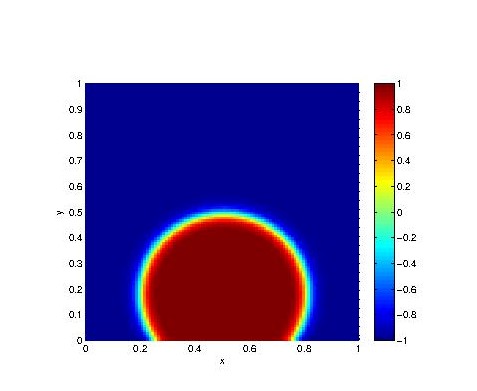}
\includegraphics[scale=0.19]{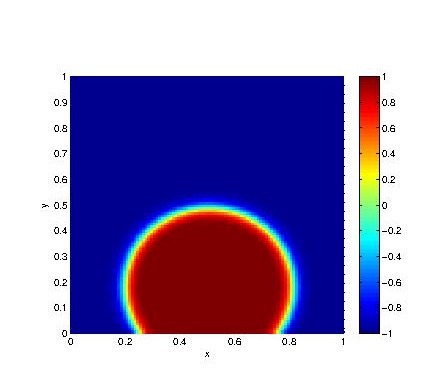}\\
\label{diffkevolution}
\caption{  Phase-field at $t=1e-4$, $t=0.01$, $t=0.04$ and $t=0.1$ with the initial data of the square shaped droplet. From top to bottom: $K=0$, $K=0.1$, $K=1$, $K=10$ and $K=\infty$.  }
\end{figure}

\begin{figure}
\centering
\includegraphics[scale=0.28]{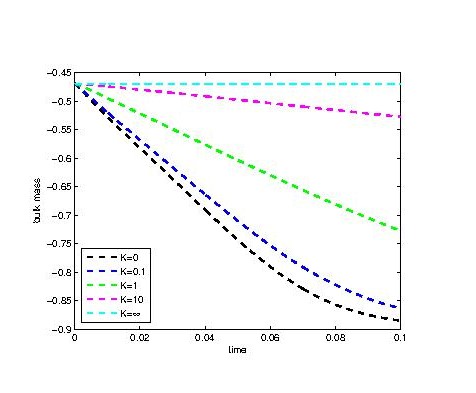}
\includegraphics[scale=0.28]{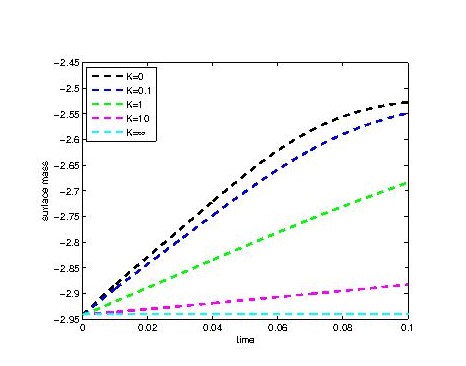}
\caption{  Time evolution of the bulk mass and the surface mass with different $K$ and the initial data of the square shaped droplet.}
\label{mass_diffk}
\end{figure}

\begin{figure}
\centering
\includegraphics[scale=0.28]{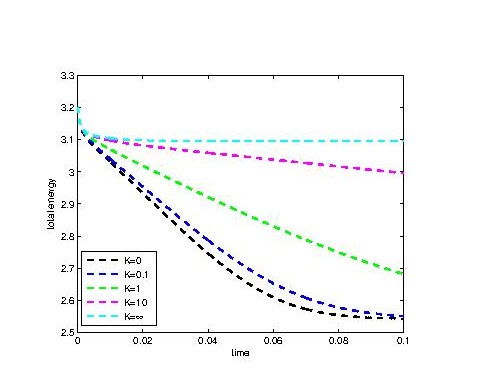}
\caption{  Time evolution of the total energy with different $K$ and the initial data of the square shaped droplet.}
\label{energy_diffk}
\end{figure}

The time evolutions of the total free energy is plotted in Fig. \ref{energy_diffk}, indicating that our numerical scheme is energy stable. We observe that an initial drop occurs for different $K$. After the initial drop, the evolution of the free energy greatly depends on $K$. When the energy in the case of $K=\infty$ stops decreasing and arrives at a stationary state, the energy still decreases for $K<\infty$. The results are consistent with the numerical results in \cite{knopf2020}.

\begin{remark}
In the numerical results above, we choose $\delta = \varepsilon$, $\kappa=1$ and the interfacial width on the boundary is the same as that in the bulk.
We notice that if we change the value of $\kappa$, the interfacial width on the boundary would not be the same as that in the bulk. Precisely, we can conclude from Fig. 17 and Fig. 18 that, since $\delta = \varepsilon$, when $\kappa > 1$ ($\kappa < 1$), the width on the boundary will be larger (smaller) than that in the bulk. In the authors' opinion, different values of $\kappa$ are related to the surface diffusion, which affects the width on the boundary.
\end{remark}
\begin{figure}
\centering
\includegraphics[scale=0.19]{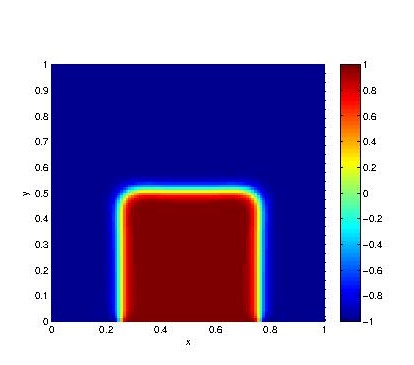}
\includegraphics[scale=0.19]{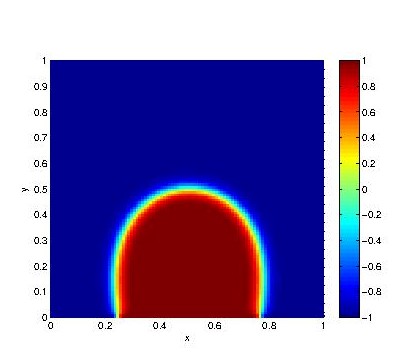}
\includegraphics[scale=0.19]{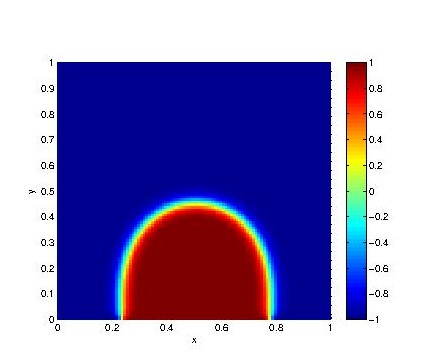}
\includegraphics[scale=0.19]{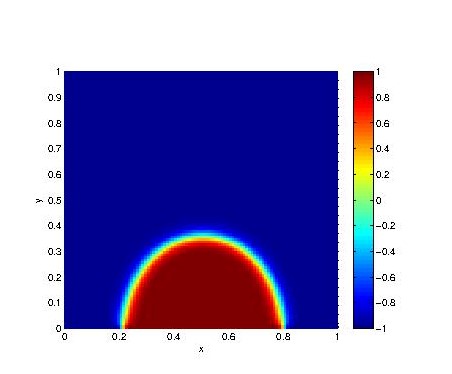}\\
\includegraphics[scale=0.19]{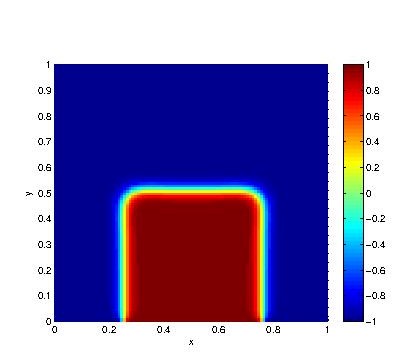}
\includegraphics[scale=0.19]{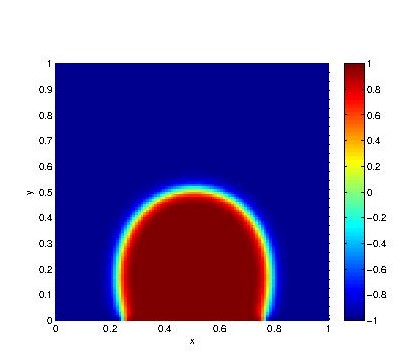}
\includegraphics[scale=0.19]{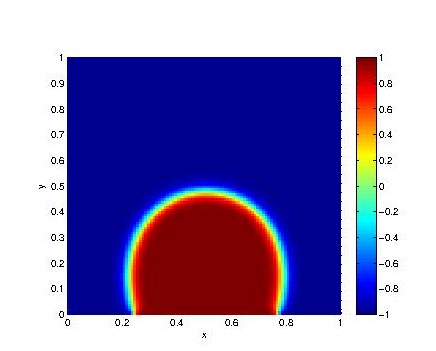}
\includegraphics[scale=0.19]{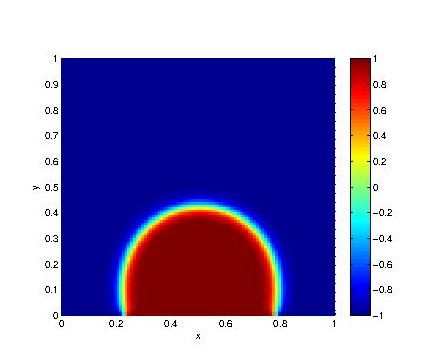}\\
\label{diffkappa1}
\caption{Phase-field at $t=1e-4$, $t=0.01$, $t=0.02$ and $t=0.04$ with the initial data of the square shaped droplet and $\kappa=0.25$: $K=0.1$ (top), $K=1$ (bottom). }
\end{figure}

\begin{figure}
\centering
\includegraphics[scale=0.19]{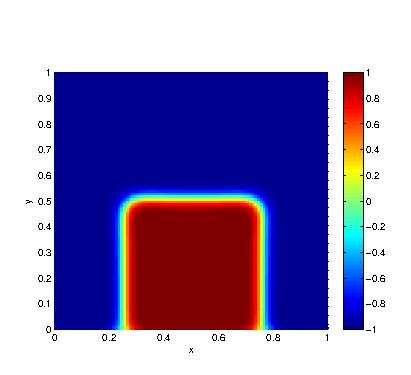}
\includegraphics[scale=0.19]{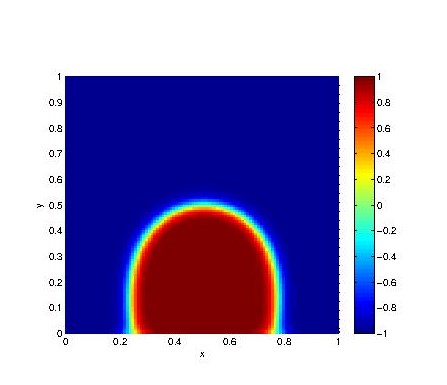}
\includegraphics[scale=0.19]{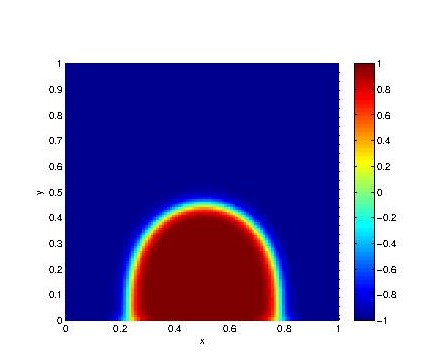}
\includegraphics[scale=0.19]{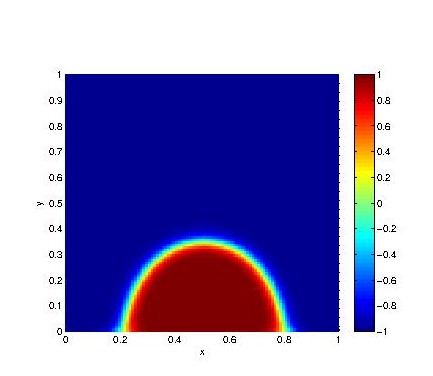}\\
\includegraphics[scale=0.19]{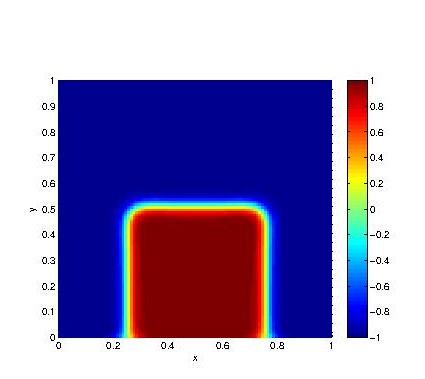}
\includegraphics[scale=0.19]{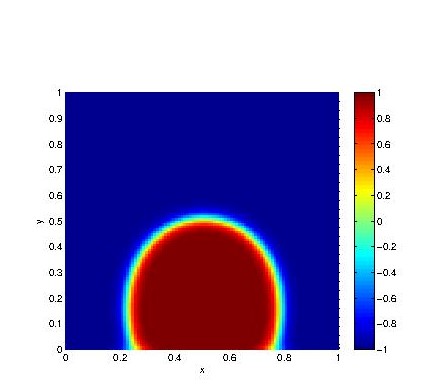}
\includegraphics[scale=0.19]{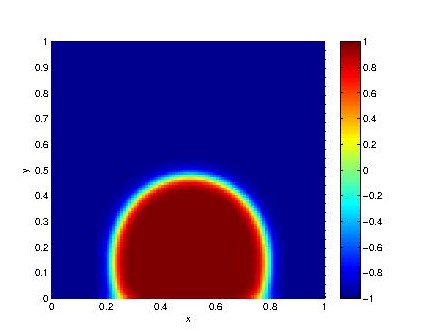}
\includegraphics[scale=0.19]{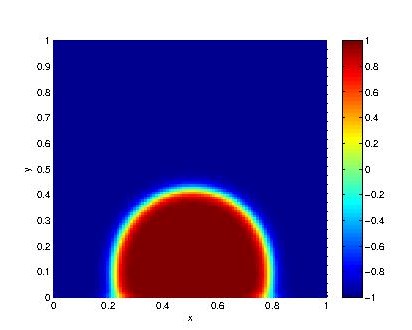}\\
\label{diffkappa2}
\caption{ Phase-field at $t=1e-4$, $t=0.01$, $t=0.02$ and $t=0.04$ with the initial data of the square shaped droplet and $\kappa=2.5$: $K=0.1$ (top), $K=1$ (bottom).  }
\end{figure}

Then we check the experimental order of convergence (EOC) of $\phi$ and $\psi$ for $K\rightarrow0$ and $K\rightarrow\infty$.
Here, the parameters are set as
$$
\varepsilon=\delta=0.02, \  \kappa=1, \  s_1=s_2=50, \ \tau=2\times1e-4.
$$
And we conduct numerical simulations from $t=0$ to $T=0.2$ with the spatial step size $h=0.01$.
Define $\phi_{*0}$ ($\psi_{*0}$) as the discrete solution under the case of $K=0$, $\phi_{*\infty}$ ($\psi_{*\infty}$) as the solution under the case of $K=\infty$ and $\phi_{K_i}$ ($\psi_{K_i}$) as the solution under the case of $K_i$.
First we compare the discrete solutions $\phi_{K_i}$ ($\psi_{K_i}$) with $\phi_{*0}$ ($\psi_{*0}$) for different $K_i$. The corresponding error is defined as
$$
Err_{i,0}=\|\phi_{K_i}-\phi_{*0}\|_{L^2(0,T; L^2(\Omega))} \ (\mbox{or}\ \|\psi_{K_i}-\psi_{*0}\|_{L^2(0,T; L^2(\Gamma))}),
$$
where the time integral is approximated using the trapezoidal rule with time increment $\tilde{\tau}=1e-3$. The experimental order is defined as
$$
EOC_{K_i}=\frac{\ln(\frac{Err_{i+1,0}}{Err_{i,0}})}{\ln(\frac{K_{i+1}}{K_{i}})}.
$$
Similarly, we can define the corresponding error and the experimental order for the case of $K\rightarrow\infty$.
The results for the convergence of $\phi$ and $\psi$ are shown in Table \ref{phito0} and Table \ref{psito0}, indicating that for $K\leq 1e-3$ and $K\geq 1e3$, the convergence rate is almost 1. The convergence rate obtained here is the same as that in \cite{knopf2020}.

\begin{table}
\centering
\begin{tabular}{c|cc} 
 K & $\|\phi_{K_i}-\phi_{*0}\|_{L^2(0,T; L^2(\Omega))}$ & EOC\\
 \hline
   1e-4 & 4.1965e-06 & -\\
   2*1e-4 & 8.3917e-06 & 0.9998\\
   5*1e-4 &2.0963e-05&0.9992\\
   1e-3 & 4.1876e-05 & 0.9983\\
   0.01 & 4.1058e-04 & 0.9914\\
   0.1 & 0.0036 & 0.9429\\
   1 & 0.0333 & 0.9661\\
\end{tabular}
\begin{tabular}{c|cc}
 K & $\|\phi_{K_i}-\phi_{*\infty}\|_{L^2(0,T; L^2(\Omega))}$ & EOC\\
 \hline
   1e4 & 1.0445e-05  & -\\
   5000 & 2.0886e-05 & -0.9997\\
   2500 & 4.1755e-05 & -0.9994\\
   2000 & 5.2182e-05 & -0.9990\\
   1000 & 1.0425e-04 & -0.9984\\
   100 & 0.0010 & -0.9819\\
   10 & 0.0086 &-0.9345\\
\end{tabular}
\caption{Comparison of $\phi$ for different $K$ with the solution for $K=0$(left) and $K=\infty$(right).}
\label{phito0}
\end{table}

\begin{table}
\centering
\begin{tabular}{c|cc}
 K & $\|\psi_{K_i}-\psi_{*0}\|_{L^2(0,T; L^2(\Gamma))}$ & EOC\\
 \hline
   1e-4 & 3.5417e-07 & -\\
   2*1e-4 & 7.0798e-07 & 0.9993\\
   5*1e-4 & 1.7681e-06 & 0.9989\\
   1e-3 & 3.5303e-06 & 0.9976\\
   0.01 & 3.4392e-05 & 0.9886\\
   0.1 & 2.9145e-04 & 0.9281\\
   1 & 0.0023 & 0.8972\\
\end{tabular}
\begin{tabular}{c|cc}
 K & $\|\psi_{K_i}-\psi_{*\infty}\|_{L^2(0,T; L^2(\Gamma))}$ & EOC\\
 \hline
   1e4 & 5.1383e-07 & -\\
   5000 & 1.0276e-06 & -0.9999\\
   2500 & 2.0549e-06 & -0.9998\\
   2000 & 2.5684e-06 & -0.9996\\
   1000 & 5.1343e-06 & -0.9993\\
   100 & 5.0937e-05 & -0.9966\\
   10 & 5.0450e-04 &-0.9958\\
\end{tabular}
\caption{Comparison of $\psi$ for different $K$ with the solution for $K=0$(left) and $K=\infty$(right).}
\label{psito0}
\end{table}

\subsection{Accuracy test}

In this section, we present numerical accuracy tests using the scheme \eqref{SIscheme1}-\eqref{SIscheme6} to support our error analysis.
Let $\Omega$ to be the unit square, the spatial step size $h=0.01$ and the parameters are chosen as $\varepsilon=\delta=0.02$, $\kappa=1$ and $s_1=s_2=50$. The initial data is set to be
\begin{equation}
\phi(x,y)=-\frac{1}2 \bigg{(} \tanh\frac{0.5-\sqrt{(x-0.5)^2+(y-0.5)^2}}{0.02}  \bigg{)}+\frac{1}2.
\end{equation}
In this section, we choose $F$ and $G$ to be the modified double-well potential \eqref{modifyFG}, and thus, the second derivative of $F$ with respect to $\phi$ and the second derivative of $G$ with respect to $\psi$ are bounded,
\begin{equation}
\max_{\phi\in\mathbb{R}} |F''(\phi)|=\max_{\psi\in\mathbb{R}} |G''(\psi)|\leq2.
\end{equation}

The errors are calculated as the difference between the solution of the coarse time step and that of the reference time step $\tau^*=10^{-5}$. In Fig. \ref{accuracytest} , we plot the $L^2$ errors of $\phi$ and $\psi$ between the numerical solution and the reference solution at $T = 0.1$ with different time step sizes in the cases of $K=1$ and $K=100$.
The results show clearly that the convergence rate of the numerical scheme
is the asymptotical at least first-order temporally for $\phi$ and $\psi$, which is consistent with our numerical analysis in Section \ref{s4}.

\begin{figure}
\centering
\includegraphics[scale=0.28]{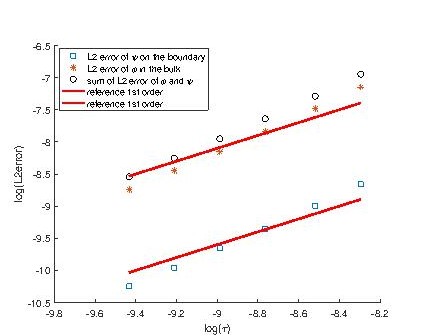}
\includegraphics[scale=0.28]{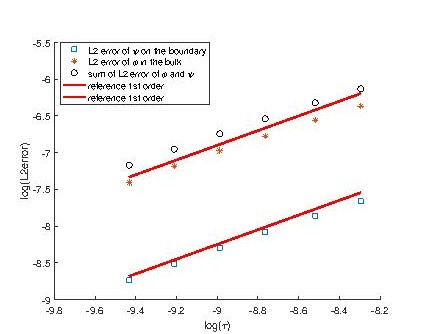}
\caption{ The $L^2$ numerical errors for $\phi$ and $\psi$ at $T = 0.1$ for $K=1$(left) and $K=100$(right).}
\label{accuracytest}
\end{figure}


\section{Conclusions}

In the present work, we consider numerical approximations and error analysis for the Cahn-Hilliard equation with reaction rate dependent dynamic boundary conditions ( P. Knopf et al., arXiv, 2020). This model can be interpreted as an interpolation between the Liu-Wu model(C. Liu and H. Wu, Arch. Rational Mech. Anal., 2019) and the GMS model(G.R. Goldstein et al., Physica D, 2011).

A first-order in time, linear and energy stable scheme
for solving this model is proposed.
The stabilization terms are utilized to enhance the stability of the scheme.
To the best of the authors' knowledge, this is the first linear and energy stable scheme for solving this new model. The semi-discretized-in-time error estimates for the scheme are also derived.

The numerical experiments are constructed in the two-dimensional space to validate the accuracy of the proposed scheme.
Moreover, the accuracy tests with respect to the time step size validate our error analysis. The convergence results for $K\rightarrow0$ and $K\rightarrow\infty$ are also illustrated, which are consistent with the former work.

%
%

\begin{center}
Acknowledgment
\end{center}

The authors would like to thank Prof. Chun Liu for some useful discussions on the subject of this article. X. Bao is thankful to Prof. Chun Liu, Prof. Yiwei Wang, Prof. Qing Cheng  and Prof. Tengfei Zhang for some stimulating discussions during the visit of Illinois Institute of Technology. X. Bao is also grateful to the Department of Applied Mathematics of Illinois Institute of Technology for the hospitality.
X. Bao is partially supported by China Scholarship Council (No. 201906040019). H. Zhang is partially supported by the National Natural Science Foundation of China (Nos. 11971002 and 11471046).

%
%



\end{document}